\documentclass[a4paper,11pt]{article}

\usepackage[defaultlines=5,all]{nowidow}		% required for avoiding pagebreaks inside paragraphs
\usepackage{indentfirst}
\usepackage[T1]{fontenc}
\usepackage[utf8]{inputenc}
\usepackage{authblk}
\usepackage{diagbox}
\usepackage{amsthm,amsmath,amssymb}
\usepackage{tikz}
\usepackage{paralist}
\usepackage[shortlabels]{enumitem}
\usepackage{subcaption}
\usepackage{hyperref}		
\usepackage{stfloats}
\usepackage{verbatim}
\usepackage{algorithm}
\usepackage{algorithmic}	
\usepackage{hhline}
\usepackage{makecell}
\usepackage{multicol}

\usepackage{tikz}
\usetikzlibrary{calc}
\usepackage{tkz-graph}

\hypersetup{
    colorlinks=true,
    citecolor = blue,
    linkcolor=blue,
    filecolor=magenta,
    urlcolor=cyan,
}

\newtheorem{theorem}{Theorem}
\newtheorem{lemma}{Lemma}
\newtheorem{claim}{Claim}
\newtheorem{rem}{Remark}

\newcommand{\red}[1]{\textcolor{black}{#1}}

\newcommand{\iain}[1]{\textcolor{black}{#1}}
\newcommand{\wen}[1]{\textcolor{black}{#1}}
\newcommand{\ben}[1]{\red{#1}}

\title{{\bf Critical $(P_5,W_4)$-Free Graphs}}
\author[a,b]{Wen Xia}
\author[c]{Jorik Jooken}
\author[c,d]{Jan Goedgebeur}
\author[e]{Iain Beaton}
\author[f]{Ben Cameron}
\author[a,b]{Shenwei Huang\thanks{Email: shenweihuang@nankai.edu.cn.}}

\affil[a]{College of Computer Science, Nankai University, Tianjin 300071, China}
\affil[b]{Tianjin Key Laboratory of Network and Data Security Technology, Nankai University, Tianjin 300071, China}
\affil[c]{Department of Computer Science, KU Leuven Campus Kulak-Kortrijk, 8500 Kortrijk, Belgium}
\affil[d]{Department of Applied Mathematics, Computer Science and Statistics, Ghent University, 9000 Ghent, Belgium}
\affil[e]{Department of Mathematics and Statistics, Acadia University, Wolfville, NS Canada}
\affil[f]{School of Mathematical and Computational Sciences, University of Prince Edward Island, Charlottetown, PE Canada}

\begin{document}

\maketitle

\begin{abstract}
A graph $G$ is $k$-vertex-critical
if $\chi(G) = k$ but $\chi(G-v)<k$ for all $v \in V(G)$. A graph is $(H_1,H_2)$-free if it contains
no induced subgraph isomorphic to $H_1$ nor $H_2$.  A $W_4$ is the graph consisting of
a $C_4$ plus an additional vertex adjacent to all the vertices of the $C_4$.
 We show that there are finitely many $k$-vertex-critical
$(P_5,W_4)$-free graphs for all $k \ge 1$ and we characterize all $5$-vertex-critical
$(P_5,W_4)$-free graphs.
Our results imply the existence of a polynomial-time certifying algorithm to decide the $k$-colorability of
$(P_5,W_4)$-free graphs for each $k \ge 1$ where the certificate is either a $k$-coloring or a $(k+1)$-vertex-critical induced subgraph.


\end{abstract}

%%%%%%%%%%%%%%%%%%%%%%%%%%%%%%
\section{Introduction}

All graphs in this paper are finite, undirected and simple. Let $P_t$ and $C_t$ denote the path and the cycle on $t$ vertices, respectively.
 Let $K_n$ be the complete graph on $n$ vertices.  The complement of $G$ is denoted by $\overline{G}$. For two graphs
$G$ and $H$, we use $G+H$ to denote the disjoint union of $G$ and $H$. For a positive integer $r$, we use $rG$
to denote the disjoint union of $r$ copies of $G$. For $s,r \geq 1$, let $K_{r,s}$ be the complete bipartite
graph with one part of size $r$ and the other part of size $s$. The {\em clique number} of $G$, denoted by $\omega(G)$, is the size of the largest clique in $G$.

	A $k$-{\em coloring} of a graph $G$ is an assignment of $k$ colors to its vertices such that adjacent vertices are assigned different
colors. A graph is $k$-{\em colorable} if it has a $k$-coloring. The {\em chromatic number} of $G$, denoted by $\chi(G)$, is the minimum number $k$ for which $G$ is $k$-colorable. A graph $G$ is said to be $k$-{\em chromatic} if $\chi(G)=k$.

We say that $G$ is {\em critical} if $\chi(H)<\chi(G)$ for every proper subgraph $H$ of $G$. A $k$-{\em critical} graph is one that is $k$-chromatic and critical. For
example, odd cycles are the only 3-critical graphs. Vertex-criticality is a weaker notion.  A graph $G$ is $k$-{\em vertex-critical}  if $\chi(G) = k$ but $\chi(G-v) < k$ for any $v \in V(G)$.

For a fixed integer $k$, the {\em $k$-coloring problem} is to determine whether a given graph is $k$-colorable. It is known
that this problem is NP-complete for all $k \ge 3$ \cite{K72}. But if we restrict the graph structure,
then there may exist polynomial-time algorithms to solve the $k$-coloring problem for all $k$.
One of the most popular structural restrictions is to forbid induced subgraphs.
A graph $G$ is $\mathcal{H}$-free if it does not contain any member in $\mathcal{H}$ as an induced subgraph. When $\mathcal{H}$ consists of a single graph $H$ or two graphs $H_1$ and
$H_2$, we write $H$-free and $(H_1,H_2)$-free instead of $\{H\}$-free and $\{H_1,H_2\}$-free, respectively. We say that $G$ is $k$-vertex-critical $\mathcal{H}$-free if it is $k$-vertex-critical
and $\mathcal{H}$-free. We say that $G$ is $k$-critical $\mathcal{H}$-free if $G$ is $k$-chromatic, $\mathcal{H}$-free and every proper $\mathcal{H}$-free subgraph of $G$ is $(k-1)$-colorable.

It is known that the $k$-coloring problem remains NP-complete on $H$-free graphs for all $k \ge 3$
when $H$ contains an induced claw~\cite{KL07} or cycle~\cite{H81,LG83}.
This means that if the $k$-coloring problem can be solved in polynomial-time on $H$-free graphs for all $k \ge 3$, then
$H$ must be a linear forest (assuming $\text{P} \neq \text{NP}$). In 2010, Ho\`{a}ng et al.~\cite{HKLSS10} proved that the $k$-coloring problem can be solved in polynomial-time
on $P_5$-free graphs for all $k$.  Later, Huang~\cite{H16} showed that the problem remains NP-complete for $P_6$-free graphs for all $k \ge 5$ and for $P_7$-free graphs for all $k \ge 4$. So $P_5$ is the largest connected subgraph that can be forbidden for which $k$-coloring can be solved
in polynomial-time for all $k$ (again assuming $\text{P} \neq \text{NP}$). Therefore, $P_5$-free graphs have attracted widespread attention from many scholars.
The algorithms in~\cite{HKLSS10}
produce a $k$-coloring if the input graph is $k$-colorable; however, the algorithms do not produce an easily verifiable certificate when the input graph is not $k$-colorable.

An algorithm is {\em certifying} if it returns with each output a simple and easily verifiable certificate that the particular output is correct. \ben{For practical purposes, certifying algorithms are highly desirable as they allow for robust yet efficient testing of implementations of the algorithms~\cite{McConnell2011}. For theoretical purposes, it is of interest to determine when certifying algorithms exist.}
For a $k$-colorability algorithm to be certifying, it
should return a $k$-coloring if one exists and a $(k+1)$-vertex-critical induced subgraph if
the graph is not $k$-colorable. \ben{In fact, when the number of $(k+1)$-vertex-critical graphs is finite in a given family, a polynomial-time $k$-colorability algorithm is readily implemented by searching the input graphs for induced $(k+1)$-vertex-critical graphs (see~\cite{CHLS19} for details).}

A \ben{linear-time} certifying algorithm for 3-colorability of $P_5$-free graphs is provided by showing that the number
of 4-vertex-critical $P_5$-free graphs is finite~\cite{BHS09,MM12}.
Given this result, one may consider whether there are only finitely many $k$-vertex-critical
$P_5$-free graphs for $k \ge 5$. Unfortunately,  \ben{for each $k\ge 5$,} it was shown that there are infinitely
many $k$-vertex-critical $P_5$-free graphs~\cite{HMRSV15}. This led to significant interest in determining which
subfamilies of $P_5$-free graphs admit polynomial-time $k$-colorability algorithms that are also certifying. 
In 2021, K.~Cameron, Goedgebeur, Huang and Shi~\cite{CGHS21} obtained the dichotomy theorem
that there are  infinitely many $k$-vertex-critical $(P_5,H)$-free graphs for $H$ of order 4 if and only if $H$ is $2P_2$ or $P_1+K_3$.
In \cite{CGHS21}, the authors also posed the natural question of which five-vertex graphs $H$ lead to only finitely many $k$-vertex-critical $(P_5,H)$-free graphs.

	Significant progress has already been made towards solving this question. \ben{The infinite family of $k$-vertex-critical $P_5$-free graphs was generalized to infinite families of $k$-vertex-critical $(P_5,C_5)$-free graphs for each $k\ge 6$~\cite{CH23}}. On the other hand, it is known that there are only finitely many $5$-vertex critical $(P_5,H)$-free graphs
 when $H=C_5$~\cite{HMRSV15}, bull~\cite{HLX23} or chair~\cite{HL23}.
  For all $k \ge 1$, it has been shown that there are only finitely many $k$-vertex-critical $(P_5,H)$-free graphs
 when
$H$ = dart~\cite{XJGH23}, banner~\cite{BGS22}, $P_2+3P_1$~\cite{CHS22}, $P_3+2P_1$~\cite{ACHS22},
            $ \overline{P_3+P_2}$ and gem~\cite{CGS}, $ K_{2,3}$ and $K_{1,4}$~\cite{KP17},
            $K_{1,3}+P_1$ and $\overline{K_3+2P_1}$~\cite{XJGH24} and $\overline{P_5}$~\cite{DHHMMP17}.

	\subsection{Our contributions}
	In this paper, we continue to study the finiteness of vertex-critical $(P_5,H)$-free graphs when $H$ has order 5. The 4-wheel $W_4$ is the graph consisting of a $C_4$ plus an additional vertex adjacent to all the vertices of the $C_4$ (see \autoref{fig:W4}). Our main result is as follows.

\begin{figure}[h!]
            \centering
            \def\c{0.5}
            \def\r{2}
             \scalebox{\c}{
            \begin{tikzpicture}
                \GraphInit[vstyle=Classic]
                \Vertex[L=\hbox{},Lpos=90,x=\r*0.0cm,y=\r*1.0cm]{v1}
                \Vertex[L=\hbox{},Lpos=180,x=\r*-0.9510565162951535cm,y=\r*0.30901699437494745cm]{v5}
                
                %\SetVertexNormal[LineWidth = 2pt]
                \Vertex[L=\hbox{},x=\r*0.9510565162951535cm,y=\r*0.30901699437494745cm]{v2} 
                \Vertex[L=\hbox{},Lpos=270,x=\r*0.0cm,y=\r*-0.38196602cm]{u}
                \Vertex[L=\hbox{},Lpos=45,x=\r*0.0cm,y=\r*0.30901699437494745cm]{v}

                \Edge[](u)(v)
                \Edge[](u)(v2)
                \Edge[](v)(v1)
                \Edge[](v)(v2)
                \Edge[](v)(v5)
                \Edge[](v1)(v5)
                \Edge[](v1)(v2)
                \Edge[](u)(v5)

            \end{tikzpicture}
            }
        \caption{The $4$-wheel $W_4$.}
            \label{fig:W4}
        \end{figure}
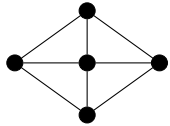

\begin{theorem}\label{th:6-vertex-critical}
   For every fixed integer $k \ge 1$, there are only finitely many $k$-vertex-critical $(P_5,W_4)$-free graphs.
\end{theorem}

Our result implies the existence of a polynomial-time certifying algorithm to decide the $k$-colorability of
$(P_5,W_4)$-free graphs for $k \ge 1$. \ben{We note that $W_4$ is the densest graph of order $5$ currently known that can be forbidden as an induced subgraph with $P_5$ and result in only finitely many $k$-vertex-critical graphs for all $k$. While the big picture of our proof techniques are similar to previous work in this area, that is, to bound the structure around odd holes and odd antiholes and apply the Strong Perfect Graph Theorem, the density of $W_4$ made it especially difficult to induce around a $C_5$. This led us to develop novel arguments involving homogeneous components of sets that we expect to be of useful for the structure of $(P_5,H)$-free graphs for $H\neq W_4$.} We also characterize all $5$-vertex-critical $(P_5,W_4)$-free graphs \ben{with the aid of exhaustive computer search} based on a recursive graph generation algorithm, leading to the following theorem:

\begin{theorem}
\label{th:characterization}
    Let $\mathcal{S}_1$ and $\mathcal{S}_2 \subset \mathcal{S}_1$ be the set of $5$-vertex-critical $(P_5,W_4)$-free graphs and $5$-critical $(P_5,W_4)$-free graphs, respectively. We have $|\mathcal{S}_1|=64$, $|\mathcal{S}_2|=21$ and the largest graphs in $\mathcal{S}_1$ and $\mathcal{S}_2$ have $17$ vertices.
\end{theorem}

	The remainder of the paper is organized as follows. We present some preliminaries in Section~\ref{Preliminarlies}. We show that there are finitely many $k$-vertex-critical ($P_5, W_4$)-free graphs for all $k \ge 1$ in Section~\ref{6-vertex-critical} \ben{assuming two lemmas that will be proved in Sections~\ref{sec:C5} and~\ref{sec:C7}, respectively}. We then characterize all such graphs for $k=5$ in Section~\ref{characterization}. Finally, we give a conclusion in Section~\ref{conclusion}.

\section{Preliminaries}\label{Preliminarlies}
For general graph theory notation we follow~\cite{BM08}. For $k\ge 4$, an induced  cycle of length $k$ is called a {\em $k$-hole}. A $k$-hole is an {\em odd hole} (respectively {\em even hole}) if $k$ is odd (respectively even). A {\em $k$-antihole} is the complement of a $k$-hole. Odd and even antiholes are defined analogously.
	
	Let $G=(V,E)$ be a graph. If $uv\in E(G)$, we say that $u$ and $v$ are {\em neighbors} or {\em adjacent}, otherwise $u$ and $v$ are {\em nonneighbors} or  {\em nonadjacent}. The {\em neighborhood} of a vertex $v$, denoted by $N_G(v)$, is the set of neighbors of $v$. For a set $X\subseteq V(G)$, let $N_G(X)=\cup_{v\in X}N_G(v)\setminus X$. We shall omit the subscript whenever the context is clear. For $x \in V(G)$ and $S\subseteq V(G)$, we denote by $N_S(x)$ the set of neighbors of $x$ that are in $S$, i.e., $N_S(x)=N_G(x)\cap S$. For two sets $X,S\subseteq V(G)$, let $N_S(X)=\cup_{v\in X}N_S(v)\setminus X$.

For $X,Y\subseteq V(G)$, we say that $X$ is {\em complete} (resp. {\em anticomplete}) to $Y$ if every vertex in $X$ is adjacent (resp. nonadjacent) to every vertex in $Y$. If $X=\{x\}$, we write ``$x$ is complete (resp. anticomplete) to $Y$'' instead of ``$\{x\}$ is complete (resp. anticomplete) to $Y$''. If a vertex $v$ is neither complete nor anticomplete to a set $S$, we say that $v$ is {\em mixed} on $S$. For a vertex $v\in V$ and an edge $xy\in E$, if $v$ is mixed on $\{x,y\}$, we say that $v$ is {\em mixed} on $xy$. For a set $H\subseteq V(G)$, if no vertex in $V(G) \setminus H$ is mixed on $H$, we say that $H$ is a {\em homogeneous set}, otherwise $H$ is a {\em nonhomogeneous set}.

    \begin{lemma}\label{lem:edgemixhomo}
    \iain{For disjoint sets $H, A \subseteq V(G)$. If $v \in A$ is not mixed on any edge in $H$, then $v$ is not mixed on any connected component of $H$.}    
    \end{lemma}

    \begin{proof}
        \iain{Suppose that $v \in A$ is not mixed on any edge in $H$.
        Let $H'$ be a connected component of $H$ with edge $xy$.
        Let $v$ be (anti)complete to $xy$.
        By the connectivity of $H'$, for any vertex $u \in H'$ there exists a path $xy_1y_2 \cdots u$ from $x$ to $u$.
        As $v$ is (non)adjacent to $x$ and $v$ is not mixed on any edge of $H'$ then $v$ is (non)adjacent to $y_1$.
        Similarly, $v$ must be (anti)complete to each edge in the path from $x$ to $u$.
        Therefore $v$ is (non)adjacent to $u$ and hence (anti)complete to $H'$.
        }
    \end{proof}

 A vertex subset $S\subseteq V(G)$ is {\em stable} if no two vertices in $S$ are adjacent. A {\em clique} is the complement of a stable set. Two nonadjacent vertices $u$ and $v$ are said to be {\em comparable} if $N(v)\subseteq N(u)$ or $N(u)\subseteq N(v)$.
 For an induced subgraph $A$ of $G$, we write $G-A$ instead of $G-V(A)$. For $S\subseteq V$, the subgraph \emph{induced} by $S$ is denoted by $G[S]$. We say that a vertex $w$ {\em distinguishes} two vertices $u$ and $v$ if $w$ is adjacent to exactly one of $u$ and $v$.

 We proceed with a few useful results that will be used later. The first folklore property of vertex-critical graphs is that such graphs
 contain no comparable vertices. A generalization of this property was presented in~\cite{CGHS21}.

     \begin{lemma}[\cite{CGHS21}] \label{lem:XY}
		Let $G$ be a $k$-vertex-critical graph. Then $G$ has no two nonempty disjoint subsets $X$ and $Y$ of $V(G)$ that satisfy all the following conditions.
		\begin{itemize}
			\item $X$ and $Y$ are anticomplete to each other.
			\item $\chi(G[X])\le\chi(G[Y])$.
			\item Y is complete to $N(X)$.
		\end{itemize}
	\end{lemma}

Another useful result is the following Lemma.

\begin{lemma}[\cite{XJGH23}]\label{lem:homoge}
Let $G$ be a $k$-vertex-critical graph and $S$ be a homogeneous set of $V(G)$.
For each component $A$ of $G[S]$, if $\chi(A) = m$ with $m < k$, then $A$ is an $m$-vertex-critical graph.
\end{lemma}

The following theorem tells us there are finitely many 4-vertex-critical $P_5$-free graphs.

\begin{theorem}[\cite{BHS09,MM12}]\label{4-vertex-cri}
If $G = (V,E)$ is a 4-vertex-critical $P_5$-free graph, then $|V| \le 13$.
\end{theorem}

 A graph $G$ is {\em perfect} if $\chi(H)=\omega(H)$ for every induced subgraph $H$ of $G$.
 We conclude this section with the famous Strong Perfect Graph Theorem.
	
	\begin{theorem}[The Strong Perfect Graph Theorem~\cite{CRST06}]\label{thm:SPGT}
		A graph is perfect if and only if it contains no odd holes or odd antiholes.
	\end{theorem}

  \section{The proof of \autoref{th:6-vertex-critical}}\label{6-vertex-critical}
   \ben{We prove \autoref{th:6-vertex-critical} assuming the following two lemmas that will be proved in Sections~\ref{sec:C5} and~\ref{sec:C7}, respectively to improve the readability of the paper.}
 \begin{lemma}\label{lem:c5}
   If $G$ \ben{is a $k$-vertex-critical $(P_5,W_4)$-free
       graph that} contains an induced $C_5$, then $G$ has finite order.
  \end{lemma}

       \begin{lemma}\label{lem:c7}
        If $G$ \ben{is a $k$-vertex-critical $(P_5,W_4)$-free
       graph that} contains an induced $\overline{C_7}$, then $G$ has finite order.
       \end{lemma}
    
\begin{proof}[Proof of \autoref{th:6-vertex-critical}]
          Let $G$ be a $k$-vertex-critical $(P_5,W_4)$-free graph. We show that $|G|$ is bounded. Let $\mathcal{L} = \{K_k,\overline{C_{2k-1}}\}$. If $G$ has a subgraph isomorphic to
  a member $L \in \mathcal{L}$, then $|V(G)|=|V(L)|$ by the definition of vertex-critical and so we are done.
  So, we assume in the following $G$ has no induced subgraph isomorphic to a member in $\mathcal{L}$. Then $G$ is imperfect \ben{since the only perfect $k$-vertex-critical graph is $K_k$}. Since $G$ is $W_4$-free, $G$ does not contain $\overline{C_{2t+1}}$ for $t \geq 4$.
  Moreover, since $G$ is $P_5$-free, it does not contain $C_{2t+1}$ for $t \geq 3$. It then follows from \autoref{thm:SPGT}, $G$ must contain $C_5$ or $\overline {C_7}$. \ben{The result now follows from \autoref{lem:c5} and \autoref{lem:c7}}.
\end{proof}

\section{Proof of \autoref{lem:c5}: Structure around an induced 5-hole}\label{sec:C5}

Let $G=(V,E)$ be a graph and $H$ be an induced subgraph of $G$.
We partition $V\setminus V(H)$ into subsets with respect to $H$ as follows:
for any $X\subseteq V(H)$, we denote by $S(X)$ the set of vertices
in $V\setminus V(H)$ that have $X$ as their \ben{neighbors in $V(H)$}, i.e.,
\[S(X)=\{v\in V\setminus V(H): N_{V(H)}(v)=X\}.\]
For $0\le m\le |V(H)|$, we denote by $S_m$ the set of vertices in $V\setminus V(H)$ that have exactly $m$
neighbors in $V(H)$. Note that $S_m=\bigcup_{X\subseteq V(H): |X|=m}S(X)$.

Let $G$ be a $(P_5,W_4)$-free graph and $C=v_1,v_2,v_3,v_4,v_5$ be an induced $C_5$ in $G$.
We partition $V\setminus C$ with respect to $C$ as follows, where all indices below are modulo five.

$S_0=\{v \in V\backslash V(C): N_C(v)=\emptyset\},$

$S_2(i)=\{v \in V\backslash V(C): N_C(v)=\{v_{i-1},v_{i+1}\}\},$

$S_3^1(i)=\{v \in V\backslash V(C): N_C(v)=\{v_{i-1},v_i,v_{i+1}\}\},$

$S_3^2(i)=\{v \in V\backslash V(C): N_C(v)=\{v_{i-2},v_i,v_{i+2}\}\},$

$S_4(i)=\{v \in V\backslash V(C): N_C(v)=\{v_{i-2},v_{i-1},v_{i+1},v_{i+2}\}\},$

$S_5=\{v \in V\backslash V(C): N_C(v)=V(C)\}.$

Let $S_3^1=\bigcup_{i=1}^5S_3^1(i)$ and $S_3^2=\bigcup_{i=1}^5S_3^2(i)$. \ben{We also overload the notation $S_2$ in the remainder of this section to let $S_2=\bigcup_{i=1}^5S_2(i)$}.
Since $G$ is $P_5$-free \ben{and any vertex adjacent to only one or two consecutive vertices on $C$ will clearly result in an induced $P_5$}, we have that

$$ V(G)=S_0\cup S_2\cup S_3^1\cup S_3^2\cup S_4\cup S_5.$$

We now prove a number of useful properties of these sets using the fact that $G$ is $(P_5,W_4)$-free. All properties are proved for $i=1$
due to symmetry.

 \begin{enumerate}[label=\bfseries (\arabic*)]

         \item $S_0$ is anticomplete to $S_2 \cup S_3^1$.\label{s0s2s31}

         Let $u \in S_0$, $v \in S_2(1) \cup S_3^1(1)$. If $uv \in E(G)$, then $\{u,v,v_2,v_3,v_4\}$ is an induced $P_5$ as in \autoref{fig:P5froms0s2s31}. 

        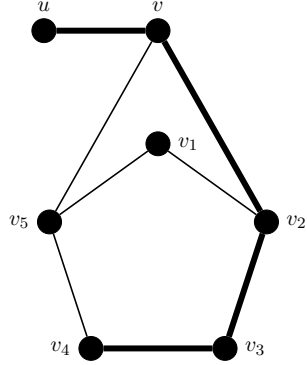
\begin{figure}[h!]
            \centering
            \def\c{0.75}
            \def\r{2}
             \scalebox{\c}{
            \begin{tikzpicture}
                \GraphInit[vstyle=Classic]
                \Vertex[L=\hbox{$v_1$},x=\r*0.0cm,y=\r*1.0cm]{v1}
                               \Vertex[L=\hbox{$v_5$},Lpos=180,x=\r*-0.9510565162951535cm,y=\r*0.30901699437494745cm]{v5}
                
                %\SetVertexNormal[LineWidth = 2pt]
                \Vertex[L=\hbox{$v_2$},x=\r*0.9510565162951535cm,y=\r*0.30901699437494745cm]{v2}
                \Vertex[L=\hbox{$v_3$},x=\r*0.5877852522924731cm,y=\r*-0.8090169943749475cm]{v3}                \Vertex[L=\hbox{$v_4$},Lpos=180,x=\r*-0.5877852522924731cm,y=\r*-0.8090169943749475cm]{v4} 
                \Vertex[L=\hbox{$v$},Lpos=90,x=\r*0.0cm,y=\r*2.0cm]{v}
                \Vertex[L=\hbox{$u$},Lpos=90,x=\r*-1.0cm,y=\r*2.0cm]{u}

                \Edge[](v1)(v5)
                \Edge[](v1)(v2)
                \Edge[](v4)(v5)
                \Edge[](v)(v5)

                \SetUpEdge[lw=3pt]
                \Edge[](u)(v)
                \Edge[](v)(v2)
                \Edge[](v2)(v3)
                \Edge[](v3)(v4)
            \end{tikzpicture}
            }
        \caption{The induced $P_5$ in bold from the proof of \ref{s0s2s31}.}
            \label{fig:P5froms0s2s31}
        \end{figure}
          \item For each $1 \leq i\leq 5$, $S_3^2(i)\cup S_4(i)$ is not mixed on any edge of $S_0$.\label{s0s32s4}

          Let $uu'$ be an edge of $S_0$, and $v \in S_3^2(1) \cup S_4(2)$. If $uv \in E(G)$ and $uv' \notin E(G)$, then $\{u',u,v,v_1,v_2\}$ is an induced $P_5$.

         \item  For each $1 \leq i \leq 5$, $S_2(i)$ is complete to $S_2(i+1)\cup S_2(i-1)\cup S_3^1(i+1) \cup S_3^1(i-1)$.\label{s2s312}

         Let $u \in S_2(1), v \in S_2(2)\cup S_3^1(2)$. If $uv \notin E(G)$, then $\{u,v_5,v_1,v,v_3\}$ is an induced $P_5$. By symmetry, $S_2(1)$ is complete to $S_2(5)\cup S_3^1(5)$.

         \item For each $1 \leq i \leq 5$, $S_2(i+2)\cup S_2(i-2) \cup S_3^2(i+1) \cup S_3^2(i-1)$  is not mixed on any edge of $S_2(i)$. \label{s21s23}

         Let $uu'$ be an edge of $S_2(1)$, and $v \in S_2(3) \cup S_3^2(2)$. If $uv \in E(G)$ and $u'v \notin E(G)$, then $\{u',u,v,v_4,v_3\}$ is an induced $P_5$ \ben{since every vertex in $S_2(3) \cup S_3^2(2)$ is adjacent to $v_4$ and nonadjacent to $v_3$}. By symmetry, $S_2(4) \cup S_3^2(5)$ is not mixed on any edge of $S_2(1)$.

         \item For each $1 \leq i \leq 5$, $S_2(i)$ is anticomplete to $S_3^1(i)$.\label{s21s31}

          Let $u \in S_2(1)$ and $v \in S_3^1(1)$. If $uv \in E(G)$, then $\{v,u,v_5,v_1,v_2\}$ is a $W_4$ as in \autoref{fig:W4froms21s31}. 
          \begin{figure}[h!]
            \centering
            \def\c{0.75}
            \def\r{2}
             \scalebox{\c}{
            \begin{tikzpicture}
                \GraphInit[vstyle=Classic]
                \Vertex[L=\hbox{$v_1$},Lpos=90,x=\r*0.0cm,y=\r*1.0cm]{v1}
                               \Vertex[L=\hbox{$v_5$},Lpos=180,x=\r*-0.9510565162951535cm,y=\r*0.30901699437494745cm]{v5}
                
                %\SetVertexNormal[LineWidth = 2pt]
                \Vertex[L=\hbox{$v_2$},x=\r*0.9510565162951535cm,y=\r*0.30901699437494745cm]{v2}
                \Vertex[L=\hbox{$v_3$},x=\r*0.5877852522924731cm,y=\r*-0.8090169943749475cm]{v3}                \Vertex[L=\hbox{$v_4$},Lpos=180,x=\r*-0.5877852522924731cm,y=\r*-0.8090169943749475cm]{v4} 
                \Vertex[L=\hbox{$u$},Lpos=270,x=\r*0.0cm,y=\r*-0.38196602cm]{u}
                \Vertex[L=\hbox{$v$},Lpos=45,x=\r*0.0cm,y=\r*0.30901699437494745cm]{v}

                \Edge[](v1)(v5)
                \Edge[](v1)(v2)
                \Edge[](v4)(v5)
                \Edge[](v2)(v3)
                \Edge[](v3)(v4)                
                
                \SetUpEdge[lw=3pt]
                \Edge[](u)(v)
                \Edge[](u)(v2)
                \Edge[](v)(v1)
                \Edge[](v)(v2)
                \Edge[](v)(v5)
                \Edge[](v1)(v5)
                \Edge[](v1)(v2)
                \Edge[](u)(v5)

            \end{tikzpicture}
            }
        \caption{The induced $W_4$ in bold from the proof of \ref{s21s31}.}
            \label{fig:W4froms21s31}
        \end{figure}

         \item For each $1 \leq i \leq 5$, $S_2(i)$ is anticomplete to $S_3^1(i+2) \cup S_3^1(i-2)$.\label{s2s313}

         Let $u \in S_2(1), v \in S_3^1(3)$. If $uv \in E(G)$, then $\{v_1,v_5,u,v,v_3\}$ is an induced $P_5$. By symmetry, $S_2(1)$ is anticomplete to $S_3^1(4)$.

         \item For each $1 \leq i \leq 5$, $S_2(i)$ is complete to $S_3^2(i)$.\label{s2(1)s32(1)}

         Let $u \in S_2(1)$ and $v \in S_3^2(1)$. If $uv \notin E(G)$, then $\{u,v_5,v_1,v,v_3\}$ is an induced $P_5$.

         \item For each $1 \leq i \leq 5$, $S_3^2(i+2) \cup S_3^2(i-2)$ is not mixed on any edge of $S_2(i)$. \label{s2(1)s32(3)}

          Let $uu'$ be an edge of $S_2(1)$, and $v \in S_3^2(3)$. If $uv \in E(G)$ and $u'v \notin E(G)$, then $\{u',u,v,v_3,v_4\}$ is an induced $P_5$. By symmetry, $S_3^2(5)$ is not mixed on any edge of $S_2(1)$.

         \item For each $1 \leq i \leq 5$ , let $u,u' \in S_2(i)$ with $uu' \notin E(G)$. Then every vertex in $\iain{S_3^2(i+1) \cup S_3^2(i-1) \cup}S_4(i) \cup S_4(i+2) \cup S_4(i-2) \cup S_5$ is  not \ben{complete to $\{u,u'\}$.}\label{s2(1)s4s5}%adjacent to at least one of $u$  and $u'$.

             Let $v \in \iain{S_3^2(5) \cup S_3^2(2) \cup}S_4(1) \cup S_4(3) \cup S_4(4) \cup S_5$. If $uv,u'v  \in E(G)$, then $\{v,u, v_2,u',v_5\}$ is a $W_4$ \ben{since $v$ is adjacent to $v_2$ and $v_5$}.

             \item For each $1 \leq i \leq 5$, $S_2(i)$ is complete to $S_4(i+1) \cup S_4(i-1)$.\label{s2(1)s4(2)}

             Let $u \in S_2(1)$ and $ v \in S_4(2)$. If $uv \notin E(G)$, then $\{u,v_2,v_1,v,v_4\}$ is an induced $P_5$.
             By symmetry, $S_2(1)$ is complete to  $S_4(5)$.

             \item For each $1 \leq i \leq 5$ , let $u,u' \in S_3^2(i)$ with $uu' \notin E(G)$. Then every vertex in $S_2(i+2)\cup S_2(i-2)$ is  not adjacent to at least one of $u$  and $u'$.\label{s32(1)s2(3)s2(4)}

                 Let $ v \in S_2(3)$. If $uv,u'v  \in E(G)$, then \iain{$\{v_4, u, v_3,u',v\}$} is a $W_4$ as in \autoref{fig:W4froms32(1)s2(3)s2(4)}. \iain{By symmetry, if $ v \in S_2(4)$ and $uv,u'v  \in E(G)$ then $\{v_3, u, v_4,u',v\}$ is a $W_4$.} 

            \begin{figure}[h!]
            \centering
            \def\c{0.75}
            \def\r{2.5}
             \scalebox{\c}{
            \begin{tikzpicture}
                \GraphInit[vstyle=Classic]
                \Vertex[L=\hbox{$v_1$},Lpos=90,x=\r*0.0cm,y=\r*1.0cm]{v1}
                               \Vertex[L=\hbox{$v_5$},Lpos=180,x=\r*-0.9510565162951535cm,y=\r*0.30901699437494745cm]{v5}
                
                %\SetVertexNormal[LineWidth = 2pt]
                \Vertex[L=\hbox{$v_2$},x=\r*0.9510565162951535cm,y=\r*0.30901699437494745cm]{v2}
                \Vertex[L=\hbox{$v_3$},Lpos=-15,x=\r*0.5877852522924731cm,y=\r*-0.8090169943749475cm]{v3}                \Vertex[L=\hbox{$v_4$},Lpos=180,x=\r*-0.5877852522924731cm,y=\r*-0.8090169943749475cm]{v4} 
                \Vertex[L=\hbox{$u$},Lpos=135,x=\r*-.25cm,y=\r*-0.38196602cm]{u}
                 % \Vertex[L=\hbox{$u'$},Lpos=270,x=\r*0.0cm,y=\r*-1.25cm]{uprime}
                 \Vertex[L=\hbox{$u'$},x=\r*1.32cm,y=\r*-0.38196602cm]{uprime}
    
                %\Vertex[L=\hbox{$v$},Lpos=90,x=\r*0.4cm,y=\r*-0.2cm]{v}
                \Vertex[L=\hbox{$v$},Lpos=90,x=\r*0.5877852522924731cm,y=\r*0cm]{v}

                \Edge[](v1)(v5)
                \Edge[](v1)(v2)
                \Edge[](v4)(v5)
                \Edge[](v2)(v3)
                \Edge[](v3)(v4)                
                \Edge[](v1)(v5)
                \Edge[](v1)(v2)
                \Edge[](u)(v1)
                \Edge[](uprime)(v1)
                \Edge[](v)(v2)
                \SetUpEdge[lw=3pt]
                \Edge[](u)(v)
                \Edge[](uprime)(v)
                
                \Edge[](u)(v3)
                \Edge[](u)(v4)
                
                \Edge[](uprime)(v3)
                \Edge[](uprime)(v4)
                \Edge[](v3)(v4)
                    
                \Edge[](v)(v4)
           
                %\tikzset{EdgeStyle/.append style = {bend right=90}}
                
                %
            \end{tikzpicture}
            }
        \caption{An induced $W_4$ in bold from the proof of \ref{s32(1)s2(3)s2(4)}.}
            \label{fig:W4froms32(1)s2(3)s2(4)}
        \end{figure}
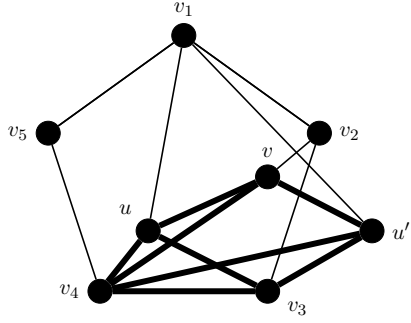

             \item $S_3^2(i)$ is complete to $S_3^1(i)$.\label{s32(1)s31(1)}

              Let $u \in S_3^2(1)$ and $v \in S_3^1(1)$. If $uv \notin E(G)$, then $\{v_2,v,v_5,v_4,u\}$ is an induced $P_5$.

              \item $S_3^2(i)$ is complete to $S_3^2(i+1) \cup S_3^2(i-1)$.\label{s32(1)s32(2)}

             Let $u \in S_3^2(1)$ and $v \in S_3^2(2)$. If $uv \notin E(G)$, then $\{u,v_3,v_2,v, v_5\}$ is an induced $P_5$.
               By symmetry, $S_3^2(1)$ is complete to  $S_3^2(5)$.

              \item $S_3^2(i)$ is complete to $S_3^1(i+2) \cup S_3^1(i-2)$.\label{s32s32(3)s32(4)}

             Let $u \in S_3^2(1)$ and $v \in S_3^1(3)$. If $uv \notin E(G)$, then $\{v,v_3,u,v_1,v_5\}$ is an induced $P_5$.
               By symmetry, $S_3^2(1)$ is complete to  $S_3^1(4)$.

              \item For each $1 \leq i \leq 5$ , let $u,u' \in S_3^2(i)$ with $uu' \notin E(G)$. Then every vertex in $S_2(i+1)\cup S_2(i-1)\cup S_3^1(i+1) \cup S_3^1(i-1) \cup S_3^2(i+2)\cup S_3^2(i-2)\cup (S_4\setminus S_4(i)) \cup S_5$ is  not adjacent to at least one of $u$  and $u'$.\label{s32s5s4}

             Let $v \in S_2(2)\cup S_3^1(2)\cup S_3^2(3) \cup S_4(2) \cup S_4(4) \cup S_5$. If $uv,u'v  \in E(G)$, then \iain{$\{v,u,v_1, u',v_3\}$} is a\ben{n induced{}} $W_4$ \ben{since $v$ is adjacent to $v_1$ and $v_3$}.
             By symmetry, every vertex in $S_2(5) \cup S_3^1(5) \cup S_3^2(4)\cup S_4(3) \cup S_4(5)$ is  not adjacent to at least one of $u$  and $u'$.

             \item For each $1 \le i \le 5$, $S_3^1(i)$, $S_4(i)$ and $S_5$ is a clique, respectively.\label{s31s4s5}

             Let $u, v \in S_3^1(1)$, $S_4(4)$ or $S_5$. If $uv \notin E(G)$, then \iain{$\{v_1,u,v_2,v,v_5\}$} is a $W_4$.

             \item For each $1 \le i \le 5$, $S_2(i)$ and $S_3^2(i)$ are $P_3$-free, respectively. \label{s2s32:p3-free}

             Let $uvw$ be a $P_3$ in $S_2(1)$ or $S_3^2(2)$, then \iain{$\{v,u,\ben{v_2},w,v_5\}$} is a $W_4$ as in \autoref{fig:W4froms2s32:p3-free}. 

            \begin{figure}[h!]
            \centering
            \def\c{0.75}
            \def\r{2}
             \scalebox{\c}{
            \begin{tikzpicture}
                \GraphInit[vstyle=Classic]
                \Vertex[L=\hbox{$v_1$},Lpos=270,x=\r*0.0cm,y=\r*1.0cm]{v1}
                               \Vertex[L=\hbox{$v_5$},Lpos=180,x=\r*-0.9510565162951535cm,y=\r*0.30901699437494745cm]{v5}
                
                %\SetVertexNormal[LineWidth = 2pt]
                \Vertex[L=\hbox{$v_2$},x=\r*0.9510565162951535cm,y=\r*0.30901699437494745cm]{v2}
                \Vertex[L=\hbox{$v_3$},x=\r*0.5877852522924731cm,y=\r*-0.8090169943749475cm]{v3}                \Vertex[L=\hbox{$v_4$},Lpos=180,x=\r*-0.5877852522924731cm,y=\r*-0.8090169943749475cm]{v4} 
                \Vertex[L=\hbox{$u$},Lpos=90,x=\r*-0.75cm,y=\r*2cm]{u}
                \Vertex[L=\hbox{$v$},Lpos=270,x=\r*0.0cm,y=\r*-0.2cm]{v}
                \Vertex[L=\hbox{$w$},Lpos=90,x=\r*0.75cm,y=\r*2cm]{w}

                \Edge[](v1)(v5)
                \Edge[](v1)(v2)
                \Edge[](v4)(v5)
                \Edge[](v)(v5)
                \Edge[](v2)(v3)
                \Edge[](v3)(v4)
                \SetUpEdge[lw=3pt]
                \Edge[](u)(v)
                \Edge[](v)(w)
                \Edge[](v)(v2)
                \Edge[](v)(v5)
                \Edge[](u)(v2)
                \Edge[](u)(v5)
                \Edge[](w)(v2)
                \Edge[](w)(v5)
                \SetUpEdge[lw=1pt,style=dashed]
                \Edge[](u)(v4)
                \Edge[](v)(v4)
                \Edge[](w)(v4)
            \end{tikzpicture}
            }
        \caption{The induced $W_4$ in bold from the proof of \ref{s2s32:p3-free} where the dashed edges are present whenever $\{u,v,w\}\subseteq S_3^2(1)$.}
            \label{fig:W4froms2s32:p3-free}
        \end{figure}
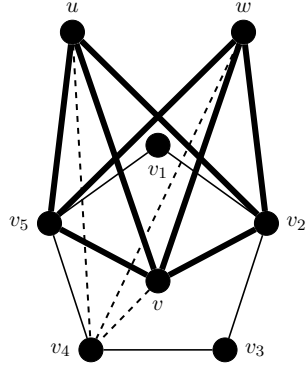

         \end{enumerate}

    \ben{With these structural properties, we are now ready to proceed with the proof of \autoref{lem:c5}.}

  \begin{proof}[Proof of \autoref{lem:c5}]
         We prove the lemma by induction on $k$. If $1 \le k \le 4$, there are finitely many $k$-vertex-critical $(P_5,W_4)$-free
       graphs by \autoref{4-vertex-cri}. In the following, we assume that $k \ge 5$ and there are finitely many $i$-vertex-critical $(P_5,W_4)$-free
       graphs for $i \le {k-1}$. 

  Let $G$ be a $k$-vertex-critical $(P_5,W_4)$-free graph such that $C=v_1,v_2,v_3,v_4,v_5$ induces a $C_5$ in $G$. We partition $V(G)$ with respect to $C$ \ben{as above in this section}. Since $G$ is $K_k$-free, combined with~\ref{s31s4s5}, we have that $S_3^1(i)$, $S_4(i)$ and $S_5$ are $K_{k-2}$-free, respectively.
So $|S_3^1(i)| \le k-3$, $|S_4(i)| \le k-3$ and $|S_5| \le k-3$. Therefore, $|S_3^1| \le 5k-15$, $|S_4| \le 5k-15$ and $|S_5| \le 5k-15$.
In the following, we will bound $S_0$, $S_2$ and $S_3^2$.

  We first bound $S_3^2(i)$.  We only give the proof of $i=1$ due to symmetry.

  \begin{claim}\label{s32(i):bounded}
  For each $1 \le i \le 5$, $S_3^2(i)$ is bounded.
  \end{claim}
  \begin{proof}
   $S=  S_3^1(2) \cup S_3^1(5) \cup (S_4 \setminus{ S_4(1)}) \cup S_5$, \red{ that is, $S$ is the union of all sets we have bounded previously that can have neighbors in $S_3^2(1)$}. Let $L = \{u| u \in S_3^2(1)\text{ and } u \text{ has a neighbor in } S\}$
  and $R = S_3^2(1) \setminus L$. Let $v \in S$ and $M_v=N(v) \cap S_3^2(1)$. By \ref{s32s5s4}, $M_v$ is a clique, so $|M_v| \le k-\ben{3}$\ben{, otherwise $k-2$ vertices of $M_v$ together with $v_3$ and $v_4$ would form a $K_k$}. Since $ L \subseteq \bigcup_{v \in  S}M_v$ and $S$ is bounded, $L$ is bounded \ben{by $|S|(k-3)$}. In the following, we will bound $R$. By \ref{s2s32:p3-free}, $S_3^2(1)$
  is $P_3$-free, so each component of $R$ is a clique. So, we only need to bound the number of components in $R$. 

  First, we will show that the number of nonhomogeneous components in $R$ is not more than $k-1$.
  Let $A_1$ and $A_2$ be two nonhomogeneous components in $R$. Then there exist $a_1, a_2 \in S_0 \cup S_2(3) \cup S_2(4) \cup S_4(1)$
  such that $a_i$ is mixed on $A_i$ where $i \in \{1,2\}$. \ben{By \autoref{lem:edgemixhomo}, we may assume $a_i$ is mixed on an edge of $A_i$.} Let $u_iw_i$ be an edge of $A_i$ and $u_ia_i \notin E(G)$, $w_ia_i \in E(G)$. \ben{If $a_i$ is mixed on $\{u_j,w_j\}$, then $\{a_i,u_i,w_i,u_j,w_j\}$ induces a $P_5$. Thus, we have the following remark that will be useful throughout the remaining proof.}
 
  \begin{rem}
      For $i,j \in \{1,2\}$, $a_i$ is complete or anticomplete to $\{u_j,w_j\}$ where $i \neq j$. \label{note:aicomporantiujwj}
  \end{rem}

  \red{From \autoref{note:aicomporantiujwj}, it follows that $a_1\neq a_2$. Moreover, if $a_1a_2 \notin E(G)$ and $a_2w_1 \not\in E(G)$ and $a_1w_2 \not\in E(G)$} then $\{a_1,w_1,v_1,w_2,a_2\}$ is an
  induced $P_5$ \red{since no vertex in $S$ is adjacent to $v_1$. Thus, we have another helpful remark.}

  \begin{rem}
      If $a_1a_2 \notin E(G)$, then $a_2w_1 \in E(G)$ or $a_1w_2 \in E(G)$. \label{note:a1a2nonadja2w1ora1w1edges}
  \end{rem}

    \ben{We now consider cases depending on which of the sets contain $a_1$ and $a_2$.}

    If $a_1 \in S_2(3) \cup S_2(4)$. Then $u_2a_1 \notin E(G)$ and $w_2a_1 \notin E(G)$ by \ref{s32(1)s2(3)s2(4)}. If $a_2 \in S_2(3) \cup S_2(4)$, then $a_2u_1 \notin E(G)$ and $a_2w_1 \notin E(G)$ by \ref{s32(1)s2(3)s2(4)}. So $a_1a_2 \in E(G)$ \ben{by the contrapositive of \autoref{note:a1a2nonadja2w1ora1w1edges}}, but now $\{u_1,w_1,a_1,a_2,w_2\}$ is an induced $P_5$.  If $a_2 \in S_0$, then $a_1a_2 \notin E(G)$ \ben{by \ref{s0s2s31}}.  So\ben{, by \autoref{note:a1a2nonadja2w1ora1w1edges}, $a_2w_1 \in E(G)$ or $a_1w_2\in E(G)$ and without loss of generality suppose} $a_2w_1 \in E(G)$. But now $\{u_2,w_2,a_2,w_1,a_1\}$ is an induced $P_5$.
    If $a_2 \in S_4(1)$, then $a_1a_2 \in E(G)$\ben{, o}therwise $\{u_2,w_2,a_2,v_2,a_1\}$ or $\{u_2,w_2,a_2,v_5,a_1\}$ is an induced $P_5$. Then $a_2w_1 \in E(G)$ \ben{or $a_2u_1 \in E(G)$}, otherwise \ben{$\{u_1,w_1,a_1,a_2,v_3\}$ induces $P_5$ if $a_1\in S_2(3)$ or $\{u_1,w_1,a_1,a_2,v_4\}$ induces $P_5$ if $a_1\in S_2(4)$. So, by \autoref{note:aicomporantiujwj}, $a_2$ is complete to $\{u_1,w_1\}$, but this contradicts \ref{s32s5s4}}.
    So $a_1 \notin S_2(3) \cup S_2(4)$. By symmetry, $a_2 \notin S_2(3) \cup S_2(4)$.

    \ben{We now show that in the remaining cases for the sets containing $a_1$ and $a_2$ (i.e., $S_0$ or $S_4(1)$) that we always have $u_1a_2, w_1a_2,u_2a_1,w_2a_1 \in E(G)$.}
    
    If $a_1,a_2 \in S_0$. If $a_1a_2 \notin E(G)$\ben{, then $a_1w_2\in E(G)$ or $a_2w_1\in E(G)$ otherwise $\{a_1,w_1,v_1,w_2,a_2\}$ is an induced $P_5$}. If $a_1w_2 \in E(G)$ and $a_2w_1 \notin E(G)$, then $a_2u_1 \notin E(G)$ \ben{by \autoref{note:aicomporantiujwj}}.
 Then $\{a_2,w_2,a_1,w_1,u_1\}$ is an induced $P_5$.
  If $a_1w_2 \in E(G)$ and $a_2w_1 \in E(G)$, then $u_1a_2 \in E(G)$ and $u_2a_1 \in E(G)$ \ben{by \autoref{note:aicomporantiujwj}}. Then $\{a_1,u_2,v_1,u_1,a_2\}$
  is an induced $P_5$. So $a_1a_2 \in E(G)$, then $u_1a_2, w_1a_2,u_2a_1,w_2a_1 \in E(G)$ by \ref{s0s32s4}.

   If $a_1,a_2 \in S_4(1)$. If $a_1a_2 \notin E(G)$\ben{, then $a_1w_2\in E(G)$ or $a_2w_1\in E(G)$ otherwise $\{a_1,w_1,v_1,w_2,a_2\}$ is an induced $P_5$}. If $a_1w_2 \in E(G)$ and $a_2w_1 \notin E(G)$, then $\{a_2,w_2,a_1,w_1,u_1\}$ is an induced $P_5$. If $a_1w_2 \in E(G)$
  and $a_2w_1 \in E(G)$, then $u_1a_2 \in E(G)$ and $u_2a_1 \in E(G)$ \ben{by \autoref{note:aicomporantiujwj}}. Then $\{a_1,u_2,v_1,u_1,a_2\}$ is an induced $P_5$. So $a_1a_2 \in E(G)$. If $a_1w_2 \notin E(G)$ and $a_2w_1 \notin E(G)$, then $\{w_2,a_2,a_1,w_1,u_1\}$ is an induced $P_5$. \ben{Thus, we must have either $a_1w_2 \notin E(G)$ or $a_2w_1 \notin E(G)$. Without loss of generality, suppose }$a_1w_2 \notin E(G)$
  and $a_2w_1 \in E(G)$\ben{. Then} $u_1a_2 \in E(G)$ and $u_2a_1 \notin E(G)$ \ben{again by \autoref{note:aicomporantiujwj}}.
  Then $\{u_2,v_1,u_1,a_2,a_1\}$ is an induced $P_5$. If $a_1w_2 \in E(G)$ and $a_2w_1 \in E(G)$, then $a_1u_2 \in E(G)$ and $a_2u_1 \in E(G)$ \ben{by \autoref{note:aicomporantiujwj}}.

  If $a_1 \in S_0$ and $a_2 \in S_4(1)$. \ben{First suppose that} $a_1a_2 \notin E(G)$.  If $a_2w_1 \in E(G)$ and $a_1w_2 \notin E(G)$, then $\{a_1,w_1,a_2,w_2,u_2\}$ is an induced $P_5$.
  If $a_2w_1 \notin E(G)$ and $a_1w_2 \in E(G)$, then $\{a_2,w_2,a_1,w_1,u_1\}$ is an induced $P_5$.
  If $a_1w_2 \in E(G)$ and $a_2w_1 \in E(G)$, then $a_1u_2 \in E(G)$ and $a_2u_1 \in E(G)$. Then $\{a_1,u_2,v_1,u_1,a_2\}$
  is an induced $P_5$. So \ben{we must have} $a_1a_2 \in E(G)$.
  If $a_1w_2 \notin E(G)$ and $a_2w_1 \notin E(G)$, then $a_1u_2 \notin E(G)$ and $a_2u_1 \notin E(G)$. Then $\{u_2,w_2,a_2,a_1,w_1\}$ is an induced $P_5$. If $a_2w_1 \notin E(G)$
  and $a_1w_2  \in E(G)$, then $\{u_1,w_1,a_1,a_2,v_2\}$ is an induced $P_5$. If $a_2w_1 \in E(G)$ and $a_1w_2 \notin E(G)$, then $\{u_2,v_1,u_1,a_2,a_1\}$ is an induced $P_5$. If $a_2w_1 \in E(G)$ and $a_1w_2 \in E(G)$, then $a_1u_2 \in E(G)$
  and $a_2u_1 \in E(G)$.

  So we have $a_1,a_2 \in S_0 \cup S_4(1)$, $a_1a_2 \in E(G)$ and $a_i$ is complete to $u_jw_j$ where $i \neq j$ and $i,j \in \{1,2\}$.
  Thus, the number of nonhomogeneous components in $R$ is not more than $k-1$, otherwise $\{a_i|a_i \in S_0  \cup S_4(1), a_i \text{ is mixed on } A_i\}$ induces a $K_k$ where $A_i$ is a nonhomogeneous component in $R$.

  Next, we will bound the number of homogeneous components in $R$. \ben{Note that if there is only one homogeneous component, then $|R|\le k-3$ since $R$ is a clique by \ref{s2s32:p3-free}. If $R$ has more than one component, then each must have neighbors outside of the set $\{v_1,v_3,v_4\}$ or else the two components contradict \autoref{lem:XY}.}
  We will first show that the number of $K_1$-components in $R$ is
   not more than $(k\iain{+1})\cdot2^{|S_3^1\cup S_4 \cup S_5|}$. Suppose not.  Then there are $k\iain{+2}$ $K_1$-components in $R$
  having the same neighbors in $S_3^1 \cup S_4 \cup S_5$ \ben{by the Pigeonhole Principle}.

  \ben{Let $\iain{U_{k+2}}=\{u_1,u_2,\ldots,\iain{u_{k+2}}\}$ be the set of $k{\wen{+2}}$ $K_1$-components of $R$ having the same neighbors in $S_3^1 \cup S_4 \cup S_5$.} \iain{We will now show that at most one vertex in $U_{k+2}$ has neighbors in $S_2(2)$ (and by symmetry $S_2(5)$).
  To show a contradiction, suppose that $u_{k+1},u_{k+2} \in U_{k+2}$ each have respective neighbors $u_{k+1}', u_{k+2}' \in S_2(2)$. By \ref{s32s5s4}, $u_{k+1}u_{k+2}', u_{k+2}u_{k+1}' \notin E(G)$. Furthermore $u_{k+1}'u_{k+2}' \in E(G)$ otherwise $\{u_{k+1}', u_{k+1}, v_4, u_{k+2}, u_{k+2}'\}$ induces a $P_5$. However, now $u_{k+1}$ and $u_{k+2}$ are mixed on the edge $u_{k+1}'u_{k+2}'$ which contradicts \ref{s21s23}. Therefore at most one vertex in $U_{k+2}$ has neighbors in $S_2(2)$. By symmetry at most one vertex in $U_{k+2}$ has neighbors in $S_2(5)$. Thus we may assume at least $k$ vertices in $U_{k+2}$ are anticomplete to $S_2(2) \cup S_2(5)$. Without loss of generality let $U=\{u_1,u_2,\ldots,u_k\}$ be anticomplete to $S_2(2) \cup S_2(5)$. Recall}
  \ben{ that each vertex in $U$ is nonadjacent, so each must have distinct neighbors or else be comparable vertices and contradict \autoref{lem:XY}. So, there exist $u_i' \in N(u_i)\setminus N(u_j)$ and $u_j' \in N(u_j)\setminus N(u_i)$.  Since $u_i$ and $u_j$ have the same neighbors in $S_3^1 \cup S_4 \cup S_5$ by assumption, it must be that $u_i',u_j' \in S_0\cup S_2\cup S_3^2$. However,  $S_2(1)$  and $S_3^2(2)\cup S_3^2(5)$  are complete to $\{u_i,u_j\}$ by \ref{s2(1)s32(1)} and \ref{s32(1)s32(2)}, respectively.  Therefore, we must have $u_i',u_j' \in S_0 \cup S_2(3) \cup S_2(4) \cup S_3^2(3) \cup S_3^2(4)$. }

  \ben{We will now show that $u_i'u_j'\in E(G)$. First note that if $\{v_1,v_3,v_4\}\not\subseteq N(u_i')\cup N(u_j')$, then $\{u_i, u_i', v_h, u_j', u_j\}$ induces a $P_5$ for some $h\in\{1,3,4\}$ unless $u_i'u_j'\in E(G)$. Thus, the only remaining cases are: }
  
  \begin{itemize}
      \item  \ben{If $u_i'\in S_3^2(3)$ and $u_j'\in S_3^2(4)$, but then $u_i'u_j'\in E(G)$ by \ref{s32(1)s32(2)}.}
      \item \ben{If $u_i'\in S_3^2(3)$ and $u_j'\in S_2(3)$, but then $u_i'u_j'\in E(G)$ by \ref{s2(1)s32(1)}.}
      \item \ben{If $u_i'\in S_2(4)$ and $u_j'\in S_3^2(4)$, but then $u_i'u_j'\in E(G)$ again by \ref{s2(1)s32(1)}.}
  \end{itemize}

\noindent \ben{Thus, $u_i'u_j'\in E(G)$.}

\ben{Now let $u_m \in U$ for some $m \neq i,j$. In the following, we will show that $u_mu_i' \notin E(G)$ and $u_mu_j' \notin E(G)$. }

  If $u_i',u_j' \in \ben{S_2(3) \cup S_2(4) \cup S_3^2(3) \cup S_3^2(4)}$ \ben{(recall that $u_i',u_j'\not\in S_2(2)\cup S_2(5)$ from above)},
  then $u_mu_i' \notin E(G)$ and $u_mu_j'\notin E(G)$ by \ref{s32(1)s2(3)s2(4)} and \ref{s32s5s4}.
  If $u_i' \in S_0$, then $u_j' \in S_3^2(3) \cup S_3^2(4)$ \ben{from \ref{s0s2s31}}. By symmetry, we assume that $u_j' \in S_3^2(3)$. \ben{To show a contradiction, suppose} $u_mu_i' \in E(G)$. Since $u_i,u_m$
  are not comparable, there exists $u_i'' \in N(u_i)\setminus N(u_m)$. If $u_i'' \in S_0$, then $u_i'u_i'' \notin E(G)$.
  Otherwise, $\{u_i'',u_i',u_m,v_1,v_2\}$ is an induced $P_5$. Then $u_i''u_j' \in E(G)$, otherwise
$\{\ben{u_i'',u_i,v_4,v_5,u_j'}\}$ is an induced $P_5$. Then $\{u_i'',u_j',u_i',u_m,v_4\}$ is an induced $P_5$.
So $u_i'' \notin S_0$. If $u_i'' \in S_3^2(3) \cup S_3^2(4)$, then $u_ju_i'' \notin E(G)$ by \ref{s32s5s4}.
So \ben{$u_i''\in N(u_i)\setminus N(u_j)$ and $u_j' \in N(u_i)\setminus N(u_j)$, and by the previous paragraph, we must have}  $u_i''u_j'\in E(G)$. If $u_i'' \in S_3^2(3)$, then $\{u_i'',u_j',u_j,v_4,v_m\}$ is an induced $P_5$.
So $u_i'' \in S_3^2(4)$. If $u_i'u_i'' \notin E(G)$, then $\{u_m,u_i',u_j',u_i'',v_2\}$
is an induced $P_5$. So $u_i'u_i'' \in E(G)$. But now $\{u_j,v_3,u_m,u_i',u_i''\}$ is an induced $P_5$.

\ben{Now for each pair of vertices $u_i,u_j\in U$, they must have vertices $u_i'$ and $u_j'$ with $u_i'\in N(u_i)\setminus N(u_j)$, $u_j' \in N(u_i)\setminus N(u_j)$, and $u_i'u_j'\in E(G)$. Further, for any $u_m\in U$ with $m\not\in \{i,j\}$, we must have $u_i'u_m\not\in E(G)$ and $u_j'u_m\not\in E(G)$. So, for any $i\in \{1,2,\dots k\}$ we must have $u_i'\in N(u_i)\setminus \left(\bigcup_{m\neq j} N(u_m)\right)$ and for all $i\neq j$, $u_i'u_j'\in E(G)$. Thus the set 
$$\left\{u_i'|u_i' \in  N(u_i)\setminus \left(\bigcup_{m\neq j} N(u_m)\right), \text { and } i \in \{1,2,\ldots,k\}\right\}$$
induces a $K_k$, contradicting $G$ being $k$-vertex-critical and not $K_k$.}

So the number of $K_1$-components in \ben{$R$} is not more than $\iain{(k+1)}\cdot2^{|S_3^1\cup S_4 \cup S_5|}$. Recall that each component of $R$ is a clique.
It is very similar\footnote{\ben{To do this, let $C_1,C_2,\ldots C_{k+2}$ be $k+2$ distinct $K_i$-components of $R$. Let $U_{k+2}=\{u_1,u_2\dots u_{k+2}\}$ be a set such that $u_j\in C_j$ and there is a vertex $u_j'\in V(G)\setminus C_j$ such that $u_j'$ is not complete to $C_{\ell}$ for some $\ell\neq j$ (which must exist by \autoref{lem:XY}). Then everything follows the same way except for the need to add ``$\setminus C_j$'' in many places and thereby adding more notation to an already notation-heavy proof.}} to show that the number of homogeneous $K_i$-components in $R$ is
 not more than $\iain{(k+1)}\cdot2^{|S_3^1\cup S_4 \cup S_5|}$ where $2 \le i \le k-3$.

Thus, $S_3^2(1)$ is bounded.
  \end{proof}

  Next, we bound $S_0$.

         \begin{claim}\label{s0:homo}
         Let $A$ be a component of $S_0$, then $A$ is homogeneous.
         \end{claim}
         \begin{proof}
         By \ref{s0s2s31} and \ref{s0s32s4}, we only need to show that each vertex of $S_5$ is complete or
         anticomplete to $A$. \ben{From \autoref{lem:edgemixhomo} it suffices to show that $S_5$ is not mixed on any edge of $A$.} Suppose $w \in S_5$ and let $uv$ be an edge of $A$ such that $wu \in E(G)$
         and $wv \notin E(G)$. Since $G$ has no clique cutset and $S_5$ is a clique \ben{from \ref{s31s4s5}}, $A$ must have neighbors in $S_3^2 \cup S_4$.

         \ben{Suppose} there exists $x \in (S_3^2 \cup S_4) \cap N(A)$\ben{. By \ref{s0s32s4}, it follows that $x$ is adjacent to both $u$ and $v$.} \ben{If} $wx \notin E(G)$, then
         there exists $v_i \in V(C)$ such that $xv_i \in E(G)$ and $xv_{i+2} \notin E(G)$.
         Then $\{v,x,v_i,w,v_{i+2}\}$ is an induced $P_5$. So $w$ is complete to $N(A) \cap (S_3^2\cup S_4)$.
         Since $G$ has no clique cutset, there exist $x_1,x_2 \in N(A)$ such that $x_1x_2 \notin E(G)$.
         Then there must exist $v_i \in V(C)$ such that $v_i \in N(x_1)\cap N(x_2)\cap N(w)$.
         But now $\{w,v_i,x_1,x_2,u\}$ is a $W_4$. Thus $A$ is  homogeneous.
        \end{proof}

        \begin{claim}\label{s0:bounded}
        $S_0$ is bounded.
        \end{claim}

        \begin{proof}
        By \autoref{s0:homo} and \autoref{lem:homoge}, each component of $S_0$ is an $m$-vertex-critical
        $(P_5,W_4)$-free graph where $1 \le m \le k-1$. By the inductive hypothesis,
         it follows that there are finitely many $m$-vertex-critical$(P_5, W_4)$-free
          graph for $1 \le m \le k-1$. By the Pigeonhole Principle, \ben{there cannot be} more than $2^{|S_3^2 \cup S_4 \cup S_5|}$ \ben{without contradicting \autoref{lem:XY}}. Thus $S_0$ is bounded.
        \end{proof}

        Finally, we bound $S_2(i)$. We only give the proof of $i=1$ due to symmetry.

        \begin{claim}
        For each $ 1\le i \le 5$, $S_2(i)$ is bounded.\label{s2(i):bounded}
        \end{claim}
        \begin{proof}
        Let $L =\{v|v \text{ has a neighbor in } S_4(1)\cup S_4(3) \cup S_4(4) \cup S_5\}$ and $R = S_2(1) \setminus L$.
        Let $ u \in S_4(1)\cup S_4(3) \cup S_4(4) \cup S_5$ and $M_u = \{v|uv \in E(G), v \in S_2(1)\}$. By \ref{s2(1)s4s5},
        $M_u$ is a clique. So $|M_u| \le k-3$. Since $ S_4(1)\cup S_4(3) \cup S_4(4) \cup S_5$ is bounded, $\bigcup_{u \in  S_4(1)\cup S_4(3) \cup S_4(4) \cup S_5}M_u$ is bounded. Note that $ L \subseteq \bigcup_{u \in S_4(1)\cup S_4(3) \cup S_4(4) \cup S_5}M_u$, and so $L$ is bounded. Next, we bound $R$.

        Let $A$ be a component in $R$. \ben{Note that if there is only one homogeneous component, then $|R|\le k-3$ since $R$ is a clique by \ref{s2s32:p3-free}. So we may assume $R$ has at least two components}. By \ref{s2s312}--\ref{s2(1)s32(3)}, \ref{s2(1)s4(2)}, \ben{and \autoref{lem:edgemixhomo}}, $A$ is homogeneous. By \ref{s2s32:p3-free} and \autoref{lem:homoge}, $A$ is a clique and $|A|\le k-2$.
         In the following, we will show that the number of $K_i$-components in $R$ is not more than $3\cdot2^{|S_3^2\setminus S_3^2(1) \cup L|}$ \ben{which is bounded by \autoref{s32(i):bounded}}. Since each component in $R$ is homogeneous,  we only need to show that the number of $K_1$-components in $R$ is not more than $3\cdot2^{|S_3^2\setminus S_3^2(1) \cup L|}$.

         Suppose not. Then there exist \ben{$K_1$-components} $u,v,w,x \in R$ such that $u,v,w,x$ have the same neighbors in $S_3^2\setminus S_3^2(1) \cup L$ \ben{by the Pigeonhole Principle}. \ben{Note that if $R$ has less than four components, we are done, so we may assume that all of $u,v,w,$ and $x$ exist. Since each is a $K_1$-component in $R$, with the same neighbors in $C$ and in $S_3^2(1)$, each must have distinct neighbors in another set or contradict \autoref{lem:XY}.}
         Let $S =\{u,v,w,x\}$. Next, we will show that for every $t_1,t_2 \in S$, $N(t_1)\cap S_2(3) \subset N(t_2)\cap S_2(3)$ or $N(t_2)\cap S_2(3) \subset N(t_1)\cap S_2(3)$.

         Suppose that there exist $t_1' \in (N(t_1) \cap S_2(3))\setminus (N(t_2) \cap S_2(3))$
         and $t_2' \in (N(t_2) \cap S_2(3))\setminus (N(t_1) \cap S_2(3))$. If $t_1't_2' \in E(G)$, then
         $\{t_1',t_2',t_2,v_5,v_1\}$ is an induced $P_5$. If $t_1't_2' \notin E(G)$, then $\{t_1',t_1,v_5,t_2,t_2'\}$ is an induced $P_5$. \ben{Thus, there cannot be $t_1' \in (N(t_1) \cap S_2(3))\setminus (N(t_2) \cap S_2(3))$
         and $t_2' \in (N(t_2) \cap S_2(3))\setminus (N(t_1) \cap S_2(3))$.}
         If $N(t_1) \cap S_2(3) = N(t_2) \cap S_2(3)$\ben{, then} since $t_1,t_2$ are not comparable, there exist $t_1' \in N(t_1)\setminus N(t_2)$
         and $t_2' \in N(t_2)\setminus N(t_1)$. \ben{By the definition of $R$ and $L$, we must have} $t_1',t_2' \in S_2(4)$. But now $\{t_2',t_1',t_1,v_2,\ben{v_1}\}$ or $\{t_2',t_2,v_2,t_1,t_1'\}$
         is an induced $P_5$ depening on whether $t_1't_2' \in E(G)$. \ben{Therefore, for every $t_1,t_2 \in S$, $N(t_1)\cap S_2(3) \subset N(t_2)\cap S_2(3)$ or $N(t_2)\cap S_2(3) \subset N(t_1)\cap S_2(3)$.} 
         
         Since $|S| =4$, \ben{we must have $$N(y_0)\cap S_2(3) \subset N(y_1)\cap S_2(3)\subset N(y_2)\cap S_2(3) $$ for some $\{\wen{y_0},y_1,y_2\}\subset S$. So,} there exist $y_1,y_2 \in S$
         such that $y_1,y_2$ have at least two common neighbors $y_1',y_2' \in S_2(3)$ and, \ben{by~\ref{s21s23}, we must have}  $y_1'y_2' \notin E(G)$.
         Then $\{y_1,y_2,y_1',y_2',v_2\}$ is a\ben{n induced} $W_4$. This is a contradiction. So $R$ is bounded.

         Thus, $S_2(1)$ is bounded.
         \end{proof}

        By \autoref{s32(i):bounded}, \autoref{s0:bounded} and \autoref{s2(i):bounded}, \autoref{lem:c5} holds.
        \end{proof}

        \section{Proof of \autoref{lem:c7}: Structure around an induced 7-antihole}\label{sec:C7}

       \begin{proof}[Proof of \autoref{lem:c7}]
       We prove the lemma by induction on $k$. If $1 \le k \le 4$, there are finitely many $k$-vertex-critical $(P_5,W_4)$-free
       graphs by \autoref{4-vertex-cri}. In the following, we assume that $k \ge 5$ and there are finitely many $i$-vertex-critical $(P_5,W_4)$-free
       graphs for $i \le {k-1}$. 
       
  Let $G$ be a $k$-vertex-critical $(P_5,W_4)$-free graph such that $C = v_1,v_2,\ldots,v_7$ induces a $\overline{C_7}$ in $G$ where $v_iv_j \in E(G)$ if and only if $\ben{1<|i - j| < 6}$ (all indices are modulo $7$).  We partition $V(G)$ with respect to $C$. \iain{Recall $S_i$ denotes the sets of vertices which have exactly $i$ neighbors in $C$. Additionally, note if $G$ has a $C_5$, then the proof is completed by \autoref{lem:c5}. So in the following, we assume that $G$ is $(P_5,W_4,C_5)$-free. }

         % If $G$ has a $C_5$, then the proof is completed by \autoref{lem:c5}. So in the following, we assume that $G$ is
       % $(P_5,W_4,C_5)$-free. Then  $S_1 = \emptyset$, $S_2 = \emptyset$,  $S_6 = \emptyset$ and $S_7 = \emptyset$.
       % So we only need to bound
       % $S_0$, $S_3$, $S_4$ and $S_5$. Moreover, we have $S_3= \bigcup_{i=1}^7 S_3(v_{i-1},v_i,v_{i+1})$, $S_4 = \bigcup_{i=1}^7 (S_4(v_{i-2},v_{i-1},v_i,v_{i+1})\cup S_4(v_{i-2},v_i,v_{i+1},v_{i+3}))$ and
       % $S_5 = \bigcup_{i=1}^7 S_5(v_{i-2},v_{i-1},v_i,v_{i+1},v_{i+3}) $.

        \iain{We begin by proving a few claims which will help show that only $S_3$ and $S_5$ are non-empty.}

        \begin{enumerate}[label=\bfseries (\Roman*)]
                       
            \item \iain{For each $1 \leq i \leq 7$, no vertex $u \notin C$ is adjacent to $v_i$, $v_{i+1}$, $v_{i-2}$, and $v_{i-3}$.} \label{c7:s1256}

            \iain{Then $\{u, v_i, v_{i+1}, v_{i-2}, v_{i-3}\}$ induces a $W_4$.}

            \item \iain{For each $1 \leq i \leq 7$, no vertex $u \notin C$ is adjacent to $v_i$, $v_{i+1}$ but not adjacent to $v_{i-1}$, $v_{i+2}$.} \label{c7:s12}

            \iain{Then $\{u, v_i, v_{i+2}, v_{i-1}, v_{i+1}\}$ induces a $C_5$.}

            \item \iain{For each $1 \leq i \leq 7$, no vertex $u \notin C$ is adjacent to $v_i$, but not adjacent to $v_{i-1}$,  $v_{i+1}$, $v_{i+2}$.} \label{c7:s1}

            \iain{Then $\{u, v_i, v_{i+2}, v_{i-1}, v_{i+1}\}$ induces a $P_5$.}
        \end{enumerate}

        \iain{It follows from \ref{c7:s1256} that $S_7=\emptyset$, $S_6=\emptyset$, and additionally the only non-empty subsets of $S_5$ are $S_5(v_{i-2},v_{i-1},v_i,v_{i+1},v_{i+3})$ for each $1 \leq i \leq 7$.
            Additionally, the only non-empty subsets of $S_4$ where \ref{c7:s1256} does not apply are $S_4(v_{i-2},v_{i-1},v_i,v_{i+1})$ and $S_4(v_{i-2},v_i,v_{i+1},v_{i+3})$. However, if $u \in S_4(v_{i-2},v_{i-1},v_i,v_{i+1})$ then $\{v_{i-2}, u, v_{i+1}, v_{i+3}, v_{i}\}$ induces a $W_4$ and if $u \in S_4(v_{i-2},v_i,v_{i+1},v_{i+3})$ then $\{u, v_{i+1}, v_{i-1}, v_{i+2}, v_i\}$ induces a $C_5$. Therefore $S_4=\emptyset$.
        From \ref{c7:s12} it follows that  the only non-empty subsets of $S_3$ are $S_3(v_{i-1},v_i,v_{i+1})$ for each $1 \leq i \leq 7$.
        To show $S_2 = \emptyset$, let $u \in S_2(v_i, v_j)$. It follows from \ref{c7:s12} that $|i-j|>1$ modulo 7. Moreover, either $i-j>2$ or $j-i>2$ modulo 7 \ben{(since if $1<i-j<2$, then it must equal 2, and therefore $j-i=-2$ which is $5$ modulo $7$)}. It then follows from \ref{c7:s1} that $S_2 = \emptyset$.
        Additionally it follows from \ref{c7:s1} that $S_1 = \emptyset$.
        Finally}, we will show that $S_0 = \emptyset$. Suppose not. By the connectivity of $G$, some vertex $u \in S_0$ has a neighbor
       $v$ \ben{and since we just showed that $S_i=\emptyset$ for all $i\in \{1,2,4,6,7\}$, we must have} $v \in S_3 \cup S_5$. It can be readily seen that there exists an index $j$ such that $v$ is adjacent to $v_j$, but not adjacent to $v_{j+1}$ or $v_{j+3}$. Then $\{u,v,v_j,v_{j+3},v_{j+1}\}$ is an induced $P_5$ \ben{(i.e., if $v\in S_5$, then from above, we know that $v\in S_5(v_{i-2},v_{i-1},v_i,v_{i+1},v_{i+3})$, so the result follows with $j=i+1$)}. This proves that $S_0 =\emptyset$.

        \ben{Thus, $S_i=\emptyset$ for all $i\neq 3,5$.} We will now show that both $S_5$ and $S_3$ are bounded.

       \begin{claim}
           \iain{For each $1 \le i \le 7$, $S_5(v_{i-2}, v_{i-1},v_i,v_{i+1}, v_{i+3})$ is bounded.}
       \end{claim}

       \begin{proof}
           \iain{We will show that $S_5(v_{i-2}, v_{i-1},v_i,v_{i+1}, v_{i+3})$ is a clique.
           Suppose not and let $u,v \in S_5(v_{i-2}, v_{i-1},v_i,v_{i+1}, v_{i+3})$ such that $uv \notin E(G)$.
           Then $\{v_{i+3}, u,\ben{v_{i}},v,v_{i+1}\}$ induces a $W_4$.
           Thus $S_5(v_{i-2}, v_{i-1},v_i,v_{i+1}, v_{i+3})$ is a clique.
           Moreover,  $v_{i-1}, v_{i+1}, v_{i+3}$ induce a $K_3$ and are complete to $S_5(v_{i-2}, v_{i-1},v_i,v_{i+1}, v_{i+3})$.
           Therefore $|S_5(v_{i-2}, v_{i-1},v_i,v_{i+1}, v_{i+3})| \le k-4$.}
       \end{proof}

       Finally, we bound $S_3$. We will first prove
       some properties about $S_3$. All properties are proved for $i=1$ due to symmetry.

        \begin{enumerate}[label=\bfseries (\Roman*)]\addtocounter{enumi}{3}
        \item For each $1 \le i \le 7$, $S_3(v_{i-1},v_i,v_{i+1})$ is
        anticomplete to $S_3(v_{i},v_{i+1},v_{i+2})\cup S_3(v_{i-2},v_{i-1},v_{i})$.\label{c7:s3(123)}

        Let $u \in S_3(v_7,v_1,v_2)$ and $v \in S_3(v_1,v_2,v_3)$. If $uv \in E(G)$, then $\{u,v,v_3,v_6,v_4\}$
        is an induced $P_5$. By symmetry, $S_3(v_7,v_1,v_2)$ is anticomplete to $S_3(v_6,v_7,v_1)$.

        \item For each $1 \le i \le 7$, $S_3(v_{i-1},v_i,v_{i+1})$ is
        anticomplete to $S_3(v_{i+1},v_{i+2},v_{i+3})\cup S_3(v_{i-3},v_{i-2},v_{i-1})$.\label{c7:s3(234)}

        Let $u \in S_3(v_7,v_1,v_2)$ and $v \in S_3(v_2,v_3,v_4)$. If $uv \in E(G)$, then $\{u,v_7,v_4,v,v_2\}$
        is a\ben{n induced} $W_4$. By symmetry, $S_3(v_7,v_1,v_2)$ is anticomplete to $S_3(v_5,v_6,v_7)$.

        \item For each $1 \le i \le 7$, $S_3(v_{i-1},v_i,v_{i+1})$ is
        anticomplete to $S_3(v_{i+2},v_{i+3},v_{i+4})\cup S_3(v_{i-4},v_{i-3},v_{i-2})$.\label{c7:s3(345)}

        Let $u \in S_3(v_7,v_1,v_2)$ and $v \in S_3(v_3,v_4,v_5)$. If $uv \in E(G)$, then $\{v_6,v_2,u,v,v_3\}$
        is a $C_5$. By symmetry, $S_3(v_7,v_1,v_2)$ is anticomplete to $S_3(v_4,v_5,v_6)$.

        \item For each $1 \le i \le 7$, $S_5 \setminus S_5(v_{i+2},v_{i+3},v_{i+4},v_{i+5},v_i) $ is not mixed on any component of $S_3(v_{i-1},v_i,v_{i+1})$.\label{c7:s3(not mix)}

        \iain{By \autoref{lem:edgemixhomo}, it suffices to show that $S_5 \setminus S_5(v_{i+2},v_{i+3},v_{i+4},v_{i+5},v_i) $ is not mixed on any edge of $S_3(v_{i-1},v_i,v_{i+1})$}.
        Let $v \in S_5 \setminus S_5(v_{3},v_{4},v_{5},v_{6},v_1)$ and suppose that $v$ is mixed on the edge $uu'$ of $S_3(v_7,v_1,v_2)$.
        Without loss of generality let $u'v \in E(G)$ and $uv \notin E(G)$. \iain{Note that $v$ is adjacent to two or three vertices in $\{v_3,v_4,v_5,v_6\}$}.
        So there must exist two adjacent vertices $v_j,v_\ell \in \{v_3,v_4,v_5,v_6\}$ such that $v$ is mixed on $\{v_j,v_\ell\}$.
        Without loss of generality let $vv_j \in E(G)$ and $vv_\ell\notin E(G)$.
        Then $\{u,u',v,v_j,v_\ell\}$ is an induced $P_5$.

         \item For each $1 \le i \le 7$, $S_3(v_{i_1},v_i,v_{i+1})$ is complete to $S_5(v_{i+2},v_{i+3},v_{i+4},v_{i+5},v_i)$.\label{c7:s3(comp)}

         Let $u \in S_3(v_7,v_1,v_2)$ and $v \in S_5(v_3,v_4,v_5,v_6,v_1)$. If $uv \notin E(G)$, then $\{u,v_2,v_4,v, v_3\}$ is an induced $P_5$.

        \end{enumerate}

        \begin{claim}
        For each $1 \le i \le 7$, $S_3(v_{i-1},v_i,v_{i+1})$ is bounded.
        \end{claim}

          \begin{proof}
          Let $A$ be a component of $S_3(v_{i-1},v_i,v_{i+1})$, then $A$ is homogeneous by \ref{c7:s3(123)}--\ref{c7:s3(comp)}.
         By \autoref{lem:homoge}, each component of $S_3(v_{i-1},v_i,v_{i+1})$ is an $m$-vertex-critical
        $(P_5,W_4)$-free graph where $1 \le m \le k-3$. By the inductive hypothesis,
         it follows that there are finitely many $m$-vertex-critical $(P_5, W_4)$-free
          graphs for $1 \le m \le k-3$. 
          \iain{Now let $H_1$ be a component of $S_3(v_{i-1},v_i,v_{i+1})$.
          Note that $N(H)=\{v_{i-1},v_i,v_{i+1}\} \cup X$ for some $X \subseteq S_5$.
          Additionally, if two components of $S_3(v_{i-1},v_i,v_{i+1})$ had the same neighborhood in $S_5$ then this would contradict \autoref{lem:XY}.
          Therefore, each component of $S_3(v_{i-1},v_i,v_{i+1})$ has a unique set of neighbors in $S_5$.}
          By the Pigeonhole Principle, $|S_3(v_{i-1},v_i,v_{i+1})| \leq 2^{|S_5|}$. 
          Thus, $S_3(v_{i-1},v_i,v_{i+1})$ is bounded.
          \end{proof}

          This completes the proof of \autoref{lem:c7}.
       \end{proof}

%By \autoref{lem:c5} and \autoref{lem:c7}, \autoref{th:6-vertex-critical} holds.

\section{Characterizing all $5$-vertex-critical $(P_5,W_4)$-free graphs}\label{characterization}

Ho\`{a}ng et al.~\cite{HKLSS10} proposed a recursive graph generation algorithm that can generate all $k$-vertex-critical $\mathcal{H}$-free graphs. This algorithm was further generalized, improved and its implementation was made more efficient and generic (by considering additional heuristics, pruning rules and parameters) in a series of papers~\cite{GS18,GJORS24,XJGH23,XJGH24}. We used this algorithm to characterize all $5$-vertex-critical $(P_5,W_4)$-free graphs.

To keep the current paper self-contained, we briefly sketch the main ideas of this algorithm. The algorithm receives as input a graph $I$ and recursively generates all $k$-vertex-critical $\mathcal{H}$-free graphs for which $I$ occurs as an induced subgraph. In order to do so, the algorithm adds a vertex $u$ to $I$ and adds edges between $u$ and vertices in $V(I)$ in all possible ways and recurses for each of the obtained graphs. However, in order to make the algorithm more efficient and even terminate in some cases\footnote{We remark that the algorithm will never terminate if there are infinitely many $k$-vertex-critical $\mathcal{H}$-free graphs, but if there are finitely many such graphs the algorithm sometimes terminates if the pruning rules are strong enough. In the latter case, it is guaranteed that there are only finitely many $k$-vertex-critical $\mathcal{H}$-free graphs and the algorithm enumerates all of them.}, several additional optimizations, heuristics and pruning rules are added while still ensuring that all $k$-vertex-critical $\mathcal{H}$-free graphs are enumerated, i.e. the algorithm is exhaustive. For example, it is clear that if the algorithm encounters a graph which is not $\mathcal{H}$-free, it does not need to recurse because such a graph cannot occur as a proper induced subgraph of any $k$-vertex-critical $\mathcal{H}$-free graph. Another example of a pruning rule is based on the following observation: if $I$ occurs as an induced subgraph of a $k$-vertex-critical $\mathcal{H}$-free graph $G$ and there are two vertices $x, y \in V(I)$ such that $N_I(x) \subseteq N_I(y)$, then there must exist a vertex $u \in V(G) \setminus V(I)$ such that $u$ is adjacent to $x$, but not to $y$ (this was in fact generalized to sets $X$ and $Y$ in \autoref{lem:XY}). This observation restricts the set of edges that can be added between $u$ and $V(I)$. For the sake of brevity, we refer the interested reader to~\cite{GS18} for a more complete overview of the different pruning rules, which are more complicated than the sketch in the current paper. When adding edges to a graph in multiple ways, often isomorphic copies will be obtained. The algorithm detects this by computing a canonical form of its input graph $I$ using the \textit{nauty} package~\cite{MP14} (two graphs are isomorphic if and only if they have the same canonical form) and prunes a graph if an isomorphic graph was already processed earlier. The pseudo code of the algorithm described in this paragraph is shown in \autoref{algo:genGraphs}.

\begin{algorithm}[ht!b]
\caption{generateGraphs(Induced graph $I$, Integer $k$, Set of graphs $\mathcal{H}$)}
\label{algo:genGraphs}
%\setstretch{0.9}
  \begin{algorithmic}[1]
		\STATE // Generate all $k$-vertex-critical $\mathcal{H}$-free graphs that contain $I$ as an induced subgraph
        \IF{$I$ can be pruned by one of the pruning rules}
            \STATE return
        \ENDIF
        
        \STATE $C \gets \text{computeCanonicalForm}(I)$

        \STATE // A graph isomorphic with $I$ was already encountered before
        \IF{$\text{allCanonicalForms}.\text{contains}(C)$}
            \STATE return
        \ELSE
            \STATE $\text{allCanonicalForms}.\text{add}(C)$
                \IF{$I$ is $k$-vertex-critical $\mathcal{H}$-free}
                    \STATE Output $I$
                \ENDIF
                \FOR{every graph $I'$ obtained by adding a vertex $u$ to $I$ and edges between $u$ and vertices in $V(I)$ in all possible ways}
    				\STATE generateGraphs($I'$,$k$,$\mathcal{H}$)
    		  \ENDFOR
        \ENDIF
  \end{algorithmic}
\end{algorithm}

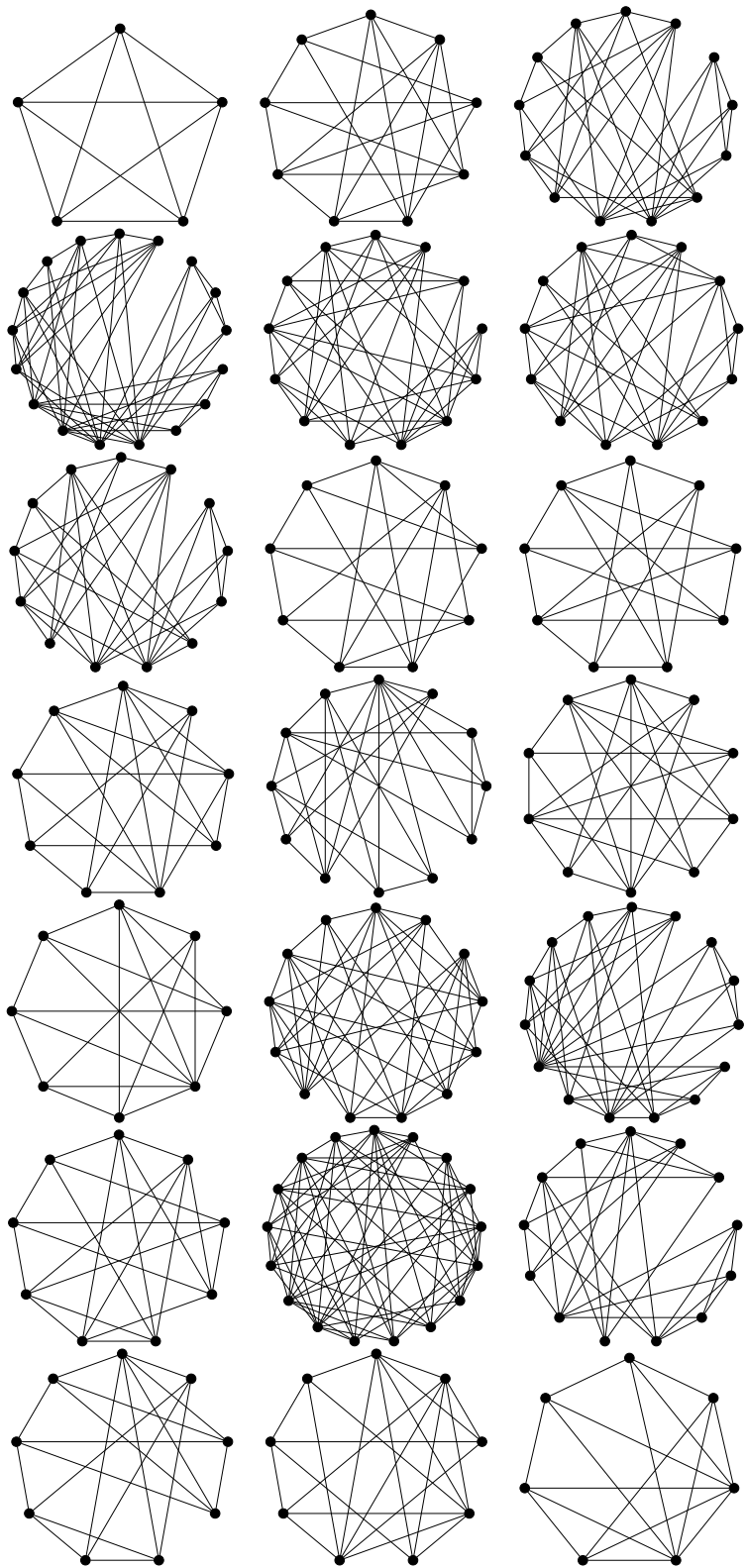
\begin{figure}[h]
    \centering
    \def\c{0.3}
    \def\r{4.75}
    
    \scalebox{\c}{\qquad\begin{tikzpicture}
\GraphInit[vstyle=Classic]
\Vertex[L=\hbox{},x=\r*0.9510565162951535cm,y=\r*0.30901699437494745cm]{v0}
\Vertex[L=\hbox{},x=\r*0.0cm,y=\r*1.0cm]{v1}
\Vertex[L=\hbox{},x=\r*-0.9510565162951535cm,y=\r*0.30901699437494745cm]{v2}
\Vertex[L=\hbox{},x=\r*-0.5877852522924731cm,y=\r*-0.8090169943749475cm]{v3}
\Vertex[L=\hbox{},x=\r*0.5877852522924731cm,y=\r*-0.8090169943749475cm]{v4}
\Edge[](v0)(v1)
\Edge[](v0)(v2)
\Edge[](v0)(v3)
\Edge[](v0)(v4)
\Edge[](v1)(v2)
\Edge[](v1)(v3)
\Edge[](v1)(v4)
\Edge[](v2)(v3)
\Edge[](v2)(v4)
\Edge[](v3)(v4)
\end{tikzpicture}}
\scalebox{\c}{\qquad\begin{tikzpicture}
\GraphInit[vstyle=Classic]
\Vertex[L=\hbox{},x=\r*0.6427876096865394cm,y=\r*0.766044443118978cm]{v0}
\Vertex[L=\hbox{},x=\r*0.0cm,y=\r*1.0cm]{v1}
\Vertex[L=\hbox{},x=\r*-0.6427876096865394cm,y=\r*0.766044443118978cm]{v2}
\Vertex[L=\hbox{},x=\r*-0.984807753012208cm,y=\r*0.17364817766693033cm]{v3}
\Vertex[L=\hbox{},x=\r*-0.8660254037844386cm,y=\r*-0.5cm]{v4}
\Vertex[L=\hbox{},x=\r*-0.3420201433256688cm,y=\r*-0.9396926207859083cm]{v5}
\Vertex[L=\hbox{},x=\r*0.3420201433256688cm,y=\r*-0.9396926207859083cm]{v6}
\Vertex[L=\hbox{},x=\r*0.8660254037844386cm,y=\r*-0.5cm]{v7}
\Vertex[L=\hbox{},x=\r*0.984807753012208cm,y=\r*0.17364817766693033cm]{v8}
\Edge[](v0)(v1)
\Edge[](v0)(v4)
\Edge[](v0)(v5)
\Edge[](v0)(v6)
\Edge[](v0)(v7)
\Edge[](v1)(v2)
\Edge[](v1)(v5)
\Edge[](v1)(v6)
\Edge[](v1)(v7)
\Edge[](v2)(v3)
\Edge[](v2)(v6)
\Edge[](v2)(v8)
\Edge[](v3)(v4)
\Edge[](v3)(v7)
\Edge[](v3)(v8)
\Edge[](v4)(v5)
\Edge[](v4)(v7)
\Edge[](v4)(v8)
\Edge[](v5)(v6)
\Edge[](v5)(v7)
\Edge[](v6)(v8)
\end{tikzpicture}}
\scalebox{\c}{\qquad\begin{tikzpicture}
\GraphInit[vstyle=Classic]
\Vertex[L=\hbox{},x=\r*0.4647231720437686cm,y=\r*0.8854560256532099cm]{v0}
\Vertex[L=\hbox{},x=\r*0.0cm,y=\r*1.0cm]{v1}
\Vertex[L=\hbox{},x=\r*-0.4647231720437686cm,y=\r*0.8854560256532099cm]{v2}
\Vertex[L=\hbox{},x=\r*-0.8229838658936564cm,y=\r*0.5680647467311558cm]{v3}
\Vertex[L=\hbox{},x=\r*-0.992708874098054cm,y=\r*0.12053668025532306cm]{v4}
\Vertex[L=\hbox{},x=\r*-0.9350162426854148cm,y=\r*-0.3546048870425356cm]{v5}
\Vertex[L=\hbox{},x=\r*-0.6631226582407953cm,y=\r*-0.7485107481711011cm]{v6}
\Vertex[L=\hbox{},x=\r*-0.23931566428755782cm,y=\r*-0.970941817426052cm]{v7}
\Vertex[L=\hbox{},x=\r*0.23931566428755782cm,y=\r*-0.970941817426052cm]{v8}
\Vertex[L=\hbox{},x=\r*0.6631226582407953cm,y=\r*-0.7485107481711011cm]{v9}
\Vertex[L=\hbox{},x=\r*0.9350162426854148cm,y=\r*-0.3546048870425356cm]{v10}
\Vertex[L=\hbox{},x=\r*0.992708874098054cm,y=\r*0.12053668025532306cm]{v11}
\Vertex[L=\hbox{},x=\r*0.8229838658936564cm,y=\r*0.5680647467311558cm]{v12}
\Edge[](v0)(v1)
\Edge[](v0)(v4)
\Edge[](v0)(v5)
\Edge[](v0)(v6)
\Edge[](v0)(v7)
\Edge[](v0)(v8)
\Edge[](v1)(v2)
\Edge[](v1)(v5)
\Edge[](v1)(v6)
\Edge[](v1)(v9)
\Edge[](v2)(v3)
\Edge[](v2)(v6)
\Edge[](v2)(v7)
\Edge[](v2)(v8)
\Edge[](v2)(v9)
\Edge[](v3)(v4)
\Edge[](v3)(v8)
\Edge[](v3)(v9)
\Edge[](v4)(v5)
\Edge[](v4)(v7)
\Edge[](v5)(v6)
\Edge[](v5)(v7)
\Edge[](v5)(v8)
\Edge[](v6)(v9)
\Edge[](v7)(v9)
\Edge[](v7)(v10)
\Edge[](v7)(v11)
\Edge[](v7)(v12)
\Edge[](v8)(v9)
\Edge[](v8)(v10)
\Edge[](v8)(v11)
\Edge[](v8)(v12)
\Edge[](v10)(v11)
\Edge[](v10)(v12)
\Edge[](v11)(v12)
\end{tikzpicture}}
\scalebox{\c}{\qquad\begin{tikzpicture}
\GraphInit[vstyle=Classic]

\Vertex[L=\hbox{},x=\r*0.36124166618715303cm,y=\r*0.9324722294043558cm]{v0}
\Vertex[L=\hbox{},x=\r*0.0cm,y=\r*1.0cm]{v1}
\Vertex[L=\hbox{},x=\r*-0.36124166618715303cm,y=\r*0.9324722294043558cm]{v2}
\Vertex[L=\hbox{},x=\r*-0.6736956436465572cm,y=\r*0.7390089172206591cm]{v3}
\Vertex[L=\hbox{},x=\r*-0.8951632913550623cm,y=\r*0.44573835577653825cm]{v4}
\Vertex[L=\hbox{},x=\r*-0.9957341762950345cm,y=\r*0.09226835946330199cm]{v5}
\Vertex[L=\hbox{},x=\r*-0.961825643172819cm,y=\r*-0.27366299007208283cm]{v6}
\Vertex[L=\hbox{},x=\r*-0.7980172272802396cm,y=\r*-0.6026346363792563cm]{v7}
\Vertex[L=\hbox{},x=\r*-0.5264321628773558cm,y=\r*-0.8502171357296141cm]{v8}
\Vertex[L=\hbox{},x=\r*-0.18374951781657053cm,y=\r*-0.9829730996839018cm]{v9}
\Vertex[L=\hbox{},x=\r*0.18374951781657053cm,y=\r*-0.9829730996839018cm]{v10}
\Vertex[L=\hbox{},x=\r*0.5264321628773558cm,y=\r*-0.8502171357296141cm]{v11}
\Vertex[L=\hbox{},x=\r*0.7980172272802396cm,y=\r*-0.6026346363792563cm]{v12}
\Vertex[L=\hbox{},x=\r*0.961825643172819cm,y=\r*-0.27366299007208283cm]{v13}
\Vertex[L=\hbox{},x=\r*0.9957341762950345cm,y=\r*0.09226835946330199cm]{v14}
\Vertex[L=\hbox{},x=\r*0.8951632913550623cm,y=\r*0.44573835577653825cm]{v15}
\Vertex[L=\hbox{},x=\r*0.6736956436465572cm,y=\r*0.7390089172206591cm]{v16}
\Edge[](v0)(v1)
\Edge[](v0)(v4)
\Edge[](v0)(v5)
\Edge[](v0)(v6)
\Edge[](v0)(v7)
\Edge[](v0)(v8)
\Edge[](v1)(v2)
\Edge[](v1)(v5)
\Edge[](v1)(v6)
\Edge[](v1)(v9)
\Edge[](v1)(v10)
\Edge[](v2)(v3)
\Edge[](v2)(v6)
\Edge[](v2)(v7)
\Edge[](v2)(v8)
\Edge[](v2)(v10)
\Edge[](v3)(v4)
\Edge[](v3)(v8)
\Edge[](v3)(v9)
\Edge[](v4)(v5)
\Edge[](v4)(v7)
\Edge[](v4)(v9)
\Edge[](v4)(v10)
\Edge[](v5)(v6)
\Edge[](v5)(v7)
\Edge[](v5)(v8)
\Edge[](v6)(v9)
\Edge[](v6)(v10)
\Edge[](v7)(v9)
\Edge[](v7)(v10)
\Edge[](v7)(v11)
\Edge[](v7)(v12)
\Edge[](v7)(v13)
\Edge[](v8)(v9)
\Edge[](v8)(v10)
\Edge[](v8)(v11)
\Edge[](v8)(v12)
\Edge[](v8)(v13)
\Edge[](v9)(v16)
\Edge[](v9)(v14)
\Edge[](v9)(v15)
\Edge[](v10)(v16)
\Edge[](v10)(v14)
\Edge[](v10)(v15)
\Edge[](v11)(v12)
\Edge[](v11)(v13)
\Edge[](v12)(v13)
\Edge[](v14)(v16)
\Edge[](v14)(v15)
\Edge[](v15)(v16)
\end{tikzpicture}}
\scalebox{\c}{\qquad\begin{tikzpicture}
\GraphInit[vstyle=Classic]
\Vertex[L=\hbox{},x=\r*0.4647231720437686cm,y=\r*0.8854560256532099cm]{v0}
\Vertex[L=\hbox{},x=\r*0.0cm,y=\r*1.0cm]{v1}
\Vertex[L=\hbox{},x=\r*-0.4647231720437686cm,y=\r*0.8854560256532099cm]{v2}
\Vertex[L=\hbox{},x=\r*-0.8229838658936564cm,y=\r*0.5680647467311558cm]{v3}
\Vertex[L=\hbox{},x=\r*-0.992708874098054cm,y=\r*0.12053668025532306cm]{v4}
\Vertex[L=\hbox{},x=\r*-0.9350162426854148cm,y=\r*-0.3546048870425356cm]{v5}
\Vertex[L=\hbox{},x=\r*-0.6631226582407953cm,y=\r*-0.7485107481711011cm]{v6}
\Vertex[L=\hbox{},x=\r*-0.23931566428755782cm,y=\r*-0.970941817426052cm]{v7}
\Vertex[L=\hbox{},x=\r*0.23931566428755782cm,y=\r*-0.970941817426052cm]{v8}
\Vertex[L=\hbox{},x=\r*0.6631226582407953cm,y=\r*-0.7485107481711011cm]{v9}
\Vertex[L=\hbox{},x=\r*0.9350162426854148cm,y=\r*-0.3546048870425356cm]{v10}
\Vertex[L=\hbox{},x=\r*0.992708874098054cm,y=\r*0.12053668025532306cm]{v11}
\Vertex[L=\hbox{},x=\r*0.8229838658936564cm,y=\r*0.5680647467311558cm]{v12}
\Edge[](v0)(v1)
\Edge[](v0)(v4)
\Edge[](v0)(v5)
\Edge[](v0)(v6)
\Edge[](v0)(v7)
\Edge[](v0)(v8)
\Edge[](v1)(v2)
\Edge[](v1)(v5)
\Edge[](v1)(v6)
\Edge[](v1)(v9)
\Edge[](v1)(v10)
\Edge[](v2)(v3)
\Edge[](v2)(v6)
\Edge[](v2)(v7)
\Edge[](v2)(v8)
\Edge[](v2)(v10)
\Edge[](v2)(v12)
\Edge[](v3)(v4)
\Edge[](v3)(v8)
\Edge[](v3)(v9)
\Edge[](v3)(v12)
\Edge[](v4)(v5)
\Edge[](v4)(v7)
\Edge[](v4)(v9)
\Edge[](v4)(v10)
\Edge[](v4)(v12)
\Edge[](v5)(v6)
\Edge[](v5)(v7)
\Edge[](v5)(v8)
\Edge[](v6)(v9)
\Edge[](v6)(v10)
\Edge[](v7)(v9)
\Edge[](v7)(v11)
\Edge[](v8)(v9)
\Edge[](v8)(v10)
\Edge[](v8)(v11)
\Edge[](v8)(v12)
\Edge[](v9)(v11)
\Edge[](v9)(v12)
\Edge[](v10)(v11)
\end{tikzpicture}}
\scalebox{\c}{\qquad\begin{tikzpicture}
\GraphInit[vstyle=Classic]
\Vertex[L=\hbox{},x=\r*0.4647231720437686cm,y=\r*0.8854560256532099cm]{v0}
\Vertex[L=\hbox{},x=\r*0.0cm,y=\r*1.0cm]{v1}
\Vertex[L=\hbox{},x=\r*-0.4647231720437686cm,y=\r*0.8854560256532099cm]{v2}
\Vertex[L=\hbox{},x=\r*-0.8229838658936564cm,y=\r*0.5680647467311558cm]{v3}
\Vertex[L=\hbox{},x=\r*-0.992708874098054cm,y=\r*0.12053668025532306cm]{v4}
\Vertex[L=\hbox{},x=\r*-0.9350162426854148cm,y=\r*-0.3546048870425356cm]{v5}
\Vertex[L=\hbox{},x=\r*-0.6631226582407953cm,y=\r*-0.7485107481711011cm]{v6}
\Vertex[L=\hbox{},x=\r*-0.23931566428755782cm,y=\r*-0.970941817426052cm]{v7}
\Vertex[L=\hbox{},x=\r*0.23931566428755782cm,y=\r*-0.970941817426052cm]{v8}
\Vertex[L=\hbox{},x=\r*0.6631226582407953cm,y=\r*-0.7485107481711011cm]{v9}
\Vertex[L=\hbox{},x=\r*0.9350162426854148cm,y=\r*-0.3546048870425356cm]{v10}
\Vertex[L=\hbox{},x=\r*0.992708874098054cm,y=\r*0.12053668025532306cm]{v11}
\Vertex[L=\hbox{},x=\r*0.8229838658936564cm,y=\r*0.5680647467311558cm]{v12}
\Edge[](v0)(v1)
\Edge[](v0)(v4)
\Edge[](v0)(v5)
\Edge[](v0)(v6)
\Edge[](v0)(v7)
\Edge[](v0)(v8)
\Edge[](v1)(v2)
\Edge[](v1)(v5)
\Edge[](v1)(v6)
\Edge[](v1)(v12)
\Edge[](v2)(v3)
\Edge[](v2)(v6)
\Edge[](v2)(v7)
\Edge[](v2)(v8)
\Edge[](v2)(v9)
\Edge[](v2)(v12)
\Edge[](v3)(v4)
\Edge[](v3)(v8)
\Edge[](v3)(v9)
\Edge[](v4)(v5)
\Edge[](v4)(v7)
\Edge[](v4)(v9)
\Edge[](v4)(v12)
\Edge[](v5)(v6)
\Edge[](v5)(v7)
\Edge[](v5)(v8)
\Edge[](v6)(v12)
\Edge[](v7)(v10)
\Edge[](v7)(v11)
\Edge[](v8)(v9)
\Edge[](v8)(v10)
\Edge[](v8)(v11)
\Edge[](v8)(v12)
\Edge[](v10)(v11)
\Edge[](v10)(v12)
\Edge[](v11)(v12)
\end{tikzpicture}}
\scalebox{\c}{\qquad\begin{tikzpicture}
\GraphInit[vstyle=Classic]
\Vertex[L=\hbox{},x=\r*0.4647231720437686cm,y=\r*0.8854560256532099cm]{v0}
\Vertex[L=\hbox{},x=\r*0.0cm,y=\r*1.0cm]{v1}
\Vertex[L=\hbox{},x=\r*-0.4647231720437686cm,y=\r*0.8854560256532099cm]{v2}
\Vertex[L=\hbox{},x=\r*-0.8229838658936564cm,y=\r*0.5680647467311558cm]{v3}
\Vertex[L=\hbox{},x=\r*-0.992708874098054cm,y=\r*0.12053668025532306cm]{v4}
\Vertex[L=\hbox{},x=\r*-0.9350162426854148cm,y=\r*-0.3546048870425356cm]{v5}
\Vertex[L=\hbox{},x=\r*-0.6631226582407953cm,y=\r*-0.7485107481711011cm]{v6}
\Vertex[L=\hbox{},x=\r*-0.23931566428755782cm,y=\r*-0.970941817426052cm]{v7}
\Vertex[L=\hbox{},x=\r*0.23931566428755782cm,y=\r*-0.970941817426052cm]{v8}
\Vertex[L=\hbox{},x=\r*0.6631226582407953cm,y=\r*-0.7485107481711011cm]{v9}
\Vertex[L=\hbox{},x=\r*0.9350162426854148cm,y=\r*-0.3546048870425356cm]{v10}
\Vertex[L=\hbox{},x=\r*0.992708874098054cm,y=\r*0.12053668025532306cm]{v11}
\Vertex[L=\hbox{},x=\r*0.8229838658936564cm,y=\r*0.5680647467311558cm]{v12}
\Edge[](v0)(v1)
\Edge[](v0)(v4)
\Edge[](v0)(v5)
\Edge[](v0)(v6)
\Edge[](v0)(v7)
\Edge[](v0)(v8)
\Edge[](v1)(v2)
\Edge[](v1)(v5)
\Edge[](v1)(v6)
\Edge[](v2)(v3)
\Edge[](v2)(v6)
\Edge[](v2)(v7)
\Edge[](v2)(v8)
\Edge[](v2)(v9)
\Edge[](v3)(v4)
\Edge[](v3)(v8)
\Edge[](v3)(v9)
\Edge[](v4)(v5)
\Edge[](v4)(v7)
\Edge[](v4)(v9)
\Edge[](v5)(v6)
\Edge[](v5)(v7)
\Edge[](v5)(v8)
\Edge[](v7)(v10)
\Edge[](v7)(v11)
\Edge[](v7)(v12)
\Edge[](v8)(v9)
\Edge[](v8)(v10)
\Edge[](v8)(v11)
\Edge[](v8)(v12)
\Edge[](v10)(v11)
\Edge[](v10)(v12)
\Edge[](v11)(v12)
\end{tikzpicture}}
\scalebox{\c}{\qquad\begin{tikzpicture}
\GraphInit[vstyle=Classic]
\Vertex[L=\hbox{},x=\r*0.6427876096865394cm,y=\r*0.766044443118978cm]{v0}
\Vertex[L=\hbox{},x=\r*0.0cm,y=\r*1.0cm]{v1}
\Vertex[L=\hbox{},x=\r*-0.6427876096865394cm,y=\r*0.766044443118978cm]{v2}
\Vertex[L=\hbox{},x=\r*-0.984807753012208cm,y=\r*0.17364817766693033cm]{v3}
\Vertex[L=\hbox{},x=\r*-0.8660254037844386cm,y=\r*-0.5cm]{v4}
\Vertex[L=\hbox{},x=\r*-0.3420201433256688cm,y=\r*-0.9396926207859083cm]{v5}
\Vertex[L=\hbox{},x=\r*0.3420201433256688cm,y=\r*-0.9396926207859083cm]{v6}
\Vertex[L=\hbox{},x=\r*0.8660254037844386cm,y=\r*-0.5cm]{v7}
\Vertex[L=\hbox{},x=\r*0.984807753012208cm,y=\r*0.17364817766693033cm]{v8}
\Edge[](v0)(v1)
\Edge[](v0)(v4)
\Edge[](v0)(v5)
\Edge[](v0)(v6)
\Edge[](v0)(v7)
\Edge[](v1)(v2)
\Edge[](v1)(v5)
\Edge[](v1)(v6)
\Edge[](v1)(v8)
\Edge[](v2)(v3)
\Edge[](v2)(v6)
\Edge[](v2)(v8)
\Edge[](v3)(v4)
\Edge[](v3)(v7)
\Edge[](v3)(v8)
\Edge[](v4)(v5)
\Edge[](v4)(v7)
\Edge[](v5)(v6)
\Edge[](v5)(v7)
\Edge[](v6)(v8)
\end{tikzpicture}}
\scalebox{\c}{\qquad\begin{tikzpicture}
\GraphInit[vstyle=Classic]
\Vertex[L=\hbox{},x=\r*0.6427876096865394cm,y=\r*0.766044443118978cm]{v0}
\Vertex[L=\hbox{},x=\r*0.0cm,y=\r*1.0cm]{v1}
\Vertex[L=\hbox{},x=\r*-0.6427876096865394cm,y=\r*0.766044443118978cm]{v2}
\Vertex[L=\hbox{},x=\r*-0.984807753012208cm,y=\r*0.17364817766693033cm]{v3}
\Vertex[L=\hbox{},x=\r*-0.8660254037844386cm,y=\r*-0.5cm]{v4}
\Vertex[L=\hbox{},x=\r*-0.3420201433256688cm,y=\r*-0.9396926207859083cm]{v5}
\Vertex[L=\hbox{},x=\r*0.3420201433256688cm,y=\r*-0.9396926207859083cm]{v6}
\Vertex[L=\hbox{},x=\r*0.8660254037844386cm,y=\r*-0.5cm]{v7}
\Vertex[L=\hbox{},x=\r*0.984807753012208cm,y=\r*0.17364817766693033cm]{v8}
\Edge[](v0)(v1)
\Edge[](v0)(v4)
\Edge[](v0)(v5)
\Edge[](v0)(v6)
\Edge[](v1)(v2)
\Edge[](v1)(v5)
\Edge[](v1)(v6)
\Edge[](v2)(v3)
\Edge[](v2)(v6)
\Edge[](v2)(v7)
\Edge[](v2)(v8)
\Edge[](v3)(v4)
\Edge[](v3)(v7)
\Edge[](v3)(v8)
\Edge[](v4)(v5)
\Edge[](v4)(v7)
\Edge[](v4)(v8)
\Edge[](v5)(v6)
\Edge[](v7)(v8)
\end{tikzpicture}}
\scalebox{\c}{\qquad\begin{tikzpicture}
\GraphInit[vstyle=Classic]
\Vertex[L=\hbox{},x=\r*0.6427876096865394cm,y=\r*0.766044443118978cm]{v0}
\Vertex[L=\hbox{},x=\r*0.0cm,y=\r*1.0cm]{v1}
\Vertex[L=\hbox{},x=\r*-0.6427876096865394cm,y=\r*0.766044443118978cm]{v2}
\Vertex[L=\hbox{},x=\r*-0.984807753012208cm,y=\r*0.17364817766693033cm]{v3}
\Vertex[L=\hbox{},x=\r*-0.8660254037844386cm,y=\r*-0.5cm]{v4}
\Vertex[L=\hbox{},x=\r*-0.3420201433256688cm,y=\r*-0.9396926207859083cm]{v5}
\Vertex[L=\hbox{},x=\r*0.3420201433256688cm,y=\r*-0.9396926207859083cm]{v6}
\Vertex[L=\hbox{},x=\r*0.8660254037844386cm,y=\r*-0.5cm]{v7}
\Vertex[L=\hbox{},x=\r*0.984807753012208cm,y=\r*0.17364817766693033cm]{v8}
\Edge[](v0)(v1)
\Edge[](v0)(v4)
\Edge[](v0)(v5)
\Edge[](v0)(v6)
\Edge[](v1)(v2)
\Edge[](v1)(v5)
\Edge[](v1)(v6)
\Edge[](v1)(v7)
\Edge[](v1)(v8)
\Edge[](v2)(v3)
\Edge[](v2)(v6)
\Edge[](v2)(v7)
\Edge[](v2)(v8)
\Edge[](v3)(v4)
\Edge[](v3)(v6)
\Edge[](v3)(v8)
\Edge[](v4)(v5)
\Edge[](v4)(v7)
\Edge[](v5)(v6)
\Edge[](v6)(v8)
\Edge[](v7)(v8)
\end{tikzpicture}}
\scalebox{\c}{\qquad\begin{tikzpicture}
\GraphInit[vstyle=Classic]
\Vertex[L=\hbox{},x=\r*0.5cm,y=\r*0.8660254037844386cm]{v0}
\Vertex[L=\hbox{},x=\r*0.0cm,y=\r*1.0cm]{v1}
\Vertex[L=\hbox{},x=\r*-0.5cm,y=\r*0.8660254037844386cm]{v2}
\Vertex[L=\hbox{},x=\r*-0.8660254037844386cm,y=\r*0.5cm]{v3}
\Vertex[L=\hbox{},x=\r*-1.0cm,y=\r*0.0cm]{v4}
\Vertex[L=\hbox{},x=\r*-0.8660254037844386cm,y=\r*-0.5cm]{v5}
\Vertex[L=\hbox{},x=\r*-0.5cm,y=\r*-0.8660254037844386cm]{v6}
\Vertex[L=\hbox{},x=\r*0.0cm,y=\r*-1.0cm]{v7}
\Vertex[L=\hbox{},x=\r*0.5cm,y=\r*-0.8660254037844386cm]{v8}
\Vertex[L=\hbox{},x=\r*0.8660254037844386cm,y=\r*-0.5cm]{v9}
\Vertex[L=\hbox{},x=\r*1.0cm,y=\r*0.0cm]{v10}
\Vertex[L=\hbox{},x=\r*0.8660254037844386cm,y=\r*0.5cm]{v11}
\Edge[](v0)(v1)
\Edge[](v0)(v4)
\Edge[](v0)(v5)
\Edge[](v0)(v6)
\Edge[](v1)(v2)
\Edge[](v1)(v5)
\Edge[](v1)(v6)
\Edge[](v1)(v7)
\Edge[](v1)(v8)
\Edge[](v1)(v9)
\Edge[](v1)(v10)
\Edge[](v1)(v11)
\Edge[](v2)(v3)
\Edge[](v2)(v6)
\Edge[](v2)(v7)
\Edge[](v2)(v8)
\Edge[](v3)(v4)
\Edge[](v3)(v6)
\Edge[](v3)(v9)
\Edge[](v3)(v10)
\Edge[](v3)(v11)
\Edge[](v4)(v5)
\Edge[](v4)(v7)
\Edge[](v4)(v8)
\Edge[](v5)(v6)
\Edge[](v7)(v8)
\Edge[](v9)(v10)
\Edge[](v9)(v11)
\Edge[](v10)(v11)
\end{tikzpicture}}
\scalebox{\c}{\qquad\begin{tikzpicture}
\GraphInit[vstyle=Classic]
\Vertex[L=\hbox{},x=\r*0.5877852522924731cm,y=\r*0.8090169943749475cm]{v0}
\Vertex[L=\hbox{},x=\r*0.0cm,y=\r*1.0cm]{v1}
\Vertex[L=\hbox{},x=\r*-0.5877852522924731cm,y=\r*0.8090169943749475cm]{v2}
\Vertex[L=\hbox{},x=\r*-0.9510565162951535cm,y=\r*0.30901699437494745cm]{v3}
\Vertex[L=\hbox{},x=\r*-0.9510565162951535cm,y=\r*-0.30901699437494745cm]{v4}
\Vertex[L=\hbox{},x=\r*-0.5877852522924731cm,y=\r*-0.8090169943749475cm]{v5}
\Vertex[L=\hbox{},x=\r*0.0cm,y=\r*-1.0cm]{v6}
\Vertex[L=\hbox{},x=\r*0.5877852522924731cm,y=\r*-0.8090169943749475cm]{v7}
\Vertex[L=\hbox{},x=\r*0.9510565162951535cm,y=\r*-0.30901699437494745cm]{v8}
\Vertex[L=\hbox{},x=\r*0.9510565162951535cm,y=\r*0.30901699437494745cm]{v9}
\Edge[](v0)(v1)
\Edge[](v0)(v4)
\Edge[](v0)(v5)
\Edge[](v0)(v6)
\Edge[](v1)(v2)
\Edge[](v1)(v5)
\Edge[](v1)(v6)
\Edge[](v1)(v7)
\Edge[](v1)(v8)
\Edge[](v2)(v3)
\Edge[](v2)(v6)
\Edge[](v2)(v7)
\Edge[](v2)(v8)
\Edge[](v2)(v9)
\Edge[](v3)(v4)
\Edge[](v3)(v6)
\Edge[](v3)(v9)
\Edge[](v4)(v5)
\Edge[](v4)(v7)
\Edge[](v4)(v8)
\Edge[](v4)(v9)
\Edge[](v5)(v6)
\Edge[](v6)(v9)
\Edge[](v7)(v8)
\end{tikzpicture}}
\scalebox{\c}{\qquad\begin{tikzpicture}
\GraphInit[vstyle=Classic]
\Vertex[L=\hbox{},x=\r*0.7071067811865476cm,y=\r*0.7071067811865476cm]{v0}
\Vertex[L=\hbox{},x=\r*0.0cm,y=\r*1.0cm]{v1}
\Vertex[L=\hbox{},x=\r*-0.7071067811865476cm,y=\r*0.7071067811865476cm]{v2}
\Vertex[L=\hbox{},x=\r*-1.0cm,y=\r*0.0cm]{v3}
\Vertex[L=\hbox{},x=\r*-0.7071067811865476cm,y=\r*-0.7071067811865476cm]{v4}
\Vertex[L=\hbox{},x=\r*0.0cm,y=\r*-1.0cm]{v5}
\Vertex[L=\hbox{},x=\r*0.7071067811865476cm,y=\r*-0.7071067811865476cm]{v6}
\Vertex[L=\hbox{},x=\r*1.0cm,y=\r*0.0cm]{v7}
\Edge[](v0)(v1)
\Edge[](v0)(v4)
\Edge[](v0)(v5)
\Edge[](v0)(v6)
\Edge[](v1)(v2)
\Edge[](v1)(v5)
\Edge[](v1)(v6)
\Edge[](v1)(v7)
\Edge[](v2)(v3)
\Edge[](v2)(v6)
\Edge[](v2)(v7)
\Edge[](v3)(v4)
\Edge[](v3)(v6)
\Edge[](v3)(v7)
\Edge[](v4)(v5)
\Edge[](v4)(v6)
\Edge[](v5)(v6)
\Edge[](v6)(v7)
\end{tikzpicture}}
\scalebox{\c}{\qquad\begin{tikzpicture}
\GraphInit[vstyle=Classic]
\Vertex[L=\hbox{},x=\r*0.4647231720437686cm,y=\r*0.8854560256532099cm]{v0}
\Vertex[L=\hbox{},x=\r*0.0cm,y=\r*1.0cm]{v1}
\Vertex[L=\hbox{},x=\r*-0.4647231720437686cm,y=\r*0.8854560256532099cm]{v2}
\Vertex[L=\hbox{},x=\r*-0.8229838658936564cm,y=\r*0.5680647467311558cm]{v3}
\Vertex[L=\hbox{},x=\r*-0.992708874098054cm,y=\r*0.12053668025532306cm]{v4}
\Vertex[L=\hbox{},x=\r*-0.9350162426854148cm,y=\r*-0.3546048870425356cm]{v5}
\Vertex[L=\hbox{},x=\r*-0.6631226582407953cm,y=\r*-0.7485107481711011cm]{v6}
\Vertex[L=\hbox{},x=\r*-0.23931566428755782cm,y=\r*-0.970941817426052cm]{v7}
\Vertex[L=\hbox{},x=\r*0.23931566428755782cm,y=\r*-0.970941817426052cm]{v8}
\Vertex[L=\hbox{},x=\r*0.6631226582407953cm,y=\r*-0.7485107481711011cm]{v9}
\Vertex[L=\hbox{},x=\r*0.9350162426854148cm,y=\r*-0.3546048870425356cm]{v10}
\Vertex[L=\hbox{},x=\r*0.992708874098054cm,y=\r*0.12053668025532306cm]{v11}
\Vertex[L=\hbox{},x=\r*0.8229838658936564cm,y=\r*0.5680647467311558cm]{v12}
\Edge[](v0)(v1)
\Edge[](v0)(v4)
\Edge[](v0)(v5)
\Edge[](v0)(v6)
\Edge[](v0)(v8)
\Edge[](v0)(v11)
\Edge[](v1)(v2)
\Edge[](v1)(v5)
\Edge[](v1)(v6)
\Edge[](v1)(v7)
\Edge[](v1)(v9)
\Edge[](v1)(v10)
\Edge[](v1)(v11)
\Edge[](v2)(v3)
\Edge[](v2)(v6)
\Edge[](v2)(v8)
\Edge[](v2)(v9)
\Edge[](v3)(v4)
\Edge[](v3)(v6)
\Edge[](v3)(v7)
\Edge[](v3)(v8)
\Edge[](v3)(v10)
\Edge[](v3)(v11)
\Edge[](v4)(v5)
\Edge[](v4)(v6)
\Edge[](v4)(v7)
\Edge[](v4)(v9)
\Edge[](v4)(v10)
\Edge[](v5)(v6)
\Edge[](v5)(v8)
\Edge[](v5)(v11)
\Edge[](v6)(v12)
\Edge[](v7)(v8)
\Edge[](v7)(v10)
\Edge[](v7)(v12)
\Edge[](v8)(v9)
\Edge[](v8)(v10)
\Edge[](v8)(v12)
\Edge[](v9)(v11)
\Edge[](v9)(v12)
\Edge[](v10)(v12)
\Edge[](v11)(v12)
\end{tikzpicture}}
\scalebox{\c}{\qquad\begin{tikzpicture}
\GraphInit[vstyle=Classic]
\Vertex[L=\hbox{},x=\r*0.4067366430758004cm,y=\r*0.9135454576426009cm]{v0}
\Vertex[L=\hbox{},x=\r*0.0cm,y=\r*1.0cm]{v1}
\Vertex[L=\hbox{},x=\r*-0.4067366430758004cm,y=\r*0.9135454576426009cm]{v2}
\Vertex[L=\hbox{},x=\r*-0.7431448254773942cm,y=\r*0.6691306063588582cm]{v3}
\Vertex[L=\hbox{},x=\r*-0.9510565162951535cm,y=\r*0.30901699437494745cm]{v4}
\Vertex[L=\hbox{},x=\r*-0.9945218953682733cm,y=\r*-0.10452846326765342cm]{v5}
\Vertex[L=\hbox{},x=\r*-0.8660254037844386cm,y=\r*-0.5cm]{v6}
\Vertex[L=\hbox{},x=\r*-0.5877852522924731cm,y=\r*-0.8090169943749475cm]{v7}
\Vertex[L=\hbox{},x=\r*-0.20791169081775923cm,y=\r*-0.9781476007338057cm]{v8}
\Vertex[L=\hbox{},x=\r*0.20791169081775923cm,y=\r*-0.9781476007338057cm]{v9}
\Vertex[L=\hbox{},x=\r*0.5877852522924731cm,y=\r*-0.8090169943749475cm]{v10}
\Vertex[L=\hbox{},x=\r*0.8660254037844386cm,y=\r*-0.5cm]{v11}
\Vertex[L=\hbox{},x=\r*0.9945218953682733cm,y=\r*-0.10452846326765342cm]{v12}
\Vertex[L=\hbox{},x=\r*0.9510565162951535cm,y=\r*0.30901699437494745cm]{v13}
\Vertex[L=\hbox{},x=\r*0.7431448254773942cm,y=\r*0.6691306063588582cm]{v14}
\Edge[](v0)(v1)
\Edge[](v0)(v4)
\Edge[](v0)(v5)
\Edge[](v0)(v6)
\Edge[](v0)(v8)
\Edge[](v1)(v2)
\Edge[](v1)(v5)
\Edge[](v1)(v6)
\Edge[](v1)(v7)
\Edge[](v1)(v9)
\Edge[](v2)(v3)
\Edge[](v2)(v6)
\Edge[](v2)(v8)
\Edge[](v2)(v9)
\Edge[](v3)(v4)
\Edge[](v3)(v6)
\Edge[](v3)(v7)
\Edge[](v3)(v8)
\Edge[](v4)(v5)
\Edge[](v4)(v6)
\Edge[](v4)(v7)
\Edge[](v4)(v9)
\Edge[](v5)(v6)
\Edge[](v5)(v8)
\Edge[](v6)(v10)
\Edge[](v6)(v11)
\Edge[](v6)(v12)
\Edge[](v6)(v13)
\Edge[](v6)(v14)
\Edge[](v7)(v8)
\Edge[](v7)(v10)
\Edge[](v7)(v11)
\Edge[](v8)(v9)
\Edge[](v8)(v12)
\Edge[](v8)(v13)
\Edge[](v8)(v14)
\Edge[](v9)(v10)
\Edge[](v9)(v11)
\Edge[](v10)(v11)
\Edge[](v12)(v13)
\Edge[](v12)(v14)
\Edge[](v13)(v14)
\end{tikzpicture}}
\scalebox{\c}{\qquad\begin{tikzpicture}
\GraphInit[vstyle=Classic]
\Vertex[L=\hbox{},x=\r*0.6427876096865394cm,y=\r*0.766044443118978cm]{v0}
\Vertex[L=\hbox{},x=\r*0.0cm,y=\r*1.0cm]{v1}
\Vertex[L=\hbox{},x=\r*-0.6427876096865394cm,y=\r*0.766044443118978cm]{v2}
\Vertex[L=\hbox{},x=\r*-0.984807753012208cm,y=\r*0.17364817766693033cm]{v3}
\Vertex[L=\hbox{},x=\r*-0.8660254037844386cm,y=\r*-0.5cm]{v4}
\Vertex[L=\hbox{},x=\r*-0.3420201433256688cm,y=\r*-0.9396926207859083cm]{v5}
\Vertex[L=\hbox{},x=\r*0.3420201433256688cm,y=\r*-0.9396926207859083cm]{v6}
\Vertex[L=\hbox{},x=\r*0.8660254037844386cm,y=\r*-0.5cm]{v7}
\Vertex[L=\hbox{},x=\r*0.984807753012208cm,y=\r*0.17364817766693033cm]{v8}
\Edge[](v0)(v1)
\Edge[](v0)(v4)
\Edge[](v0)(v5)
\Edge[](v0)(v6)
\Edge[](v0)(v7)
\Edge[](v1)(v2)
\Edge[](v1)(v5)
\Edge[](v1)(v6)
\Edge[](v1)(v7)
\Edge[](v2)(v3)
\Edge[](v2)(v6)
\Edge[](v2)(v8)
\Edge[](v3)(v4)
\Edge[](v3)(v7)
\Edge[](v3)(v8)
\Edge[](v4)(v5)
\Edge[](v4)(v6)
\Edge[](v4)(v8)
\Edge[](v5)(v6)
\Edge[](v5)(v7)
\Edge[](v7)(v8)
\end{tikzpicture}}
\scalebox{\c}{\qquad\begin{tikzpicture}
\GraphInit[vstyle=Classic]
\Vertex[L=\hbox{},x=\r*0.36124166618715303cm,y=\r*0.9324722294043558cm]{v0}
\Vertex[L=\hbox{},x=\r*0.0cm,y=\r*1.0cm]{v1}
\Vertex[L=\hbox{},x=\r*-0.36124166618715303cm,y=\r*0.9324722294043558cm]{v2}
\Vertex[L=\hbox{},x=\r*-0.6736956436465572cm,y=\r*0.7390089172206591cm]{v3}
\Vertex[L=\hbox{},x=\r*-0.8951632913550623cm,y=\r*0.44573835577653825cm]{v4}
\Vertex[L=\hbox{},x=\r*-0.9957341762950345cm,y=\r*0.09226835946330199cm]{v5}
\Vertex[L=\hbox{},x=\r*-0.961825643172819cm,y=\r*-0.27366299007208283cm]{v6}
\Vertex[L=\hbox{},x=\r*-0.7980172272802396cm,y=\r*-0.6026346363792563cm]{v7}
\Vertex[L=\hbox{},x=\r*-0.5264321628773558cm,y=\r*-0.8502171357296141cm]{v8}
\Vertex[L=\hbox{},x=\r*-0.18374951781657053cm,y=\r*-0.9829730996839018cm]{v9}
\Vertex[L=\hbox{},x=\r*0.18374951781657053cm,y=\r*-0.9829730996839018cm]{v10}
\Vertex[L=\hbox{},x=\r*0.5264321628773558cm,y=\r*-0.8502171357296141cm]{v11}
\Vertex[L=\hbox{},x=\r*0.7980172272802396cm,y=\r*-0.6026346363792563cm]{v12}
\Vertex[L=\hbox{},x=\r*0.961825643172819cm,y=\r*-0.27366299007208283cm]{v13}
\Vertex[L=\hbox{},x=\r*0.9957341762950345cm,y=\r*0.09226835946330199cm]{v14}
\Vertex[L=\hbox{},x=\r*0.8951632913550623cm,y=\r*0.44573835577653825cm]{v15}
\Vertex[L=\hbox{},x=\r*0.6736956436465572cm,y=\r*0.7390089172206591cm]{v16}
\Edge[](v0)(v1)
\Edge[](v0)(v4)
\Edge[](v0)(v5)
\Edge[](v0)(v6)
\Edge[](v0)(v7)
\Edge[](v0)(v8)
\Edge[](v0)(v13)
\Edge[](v1)(v16)
\Edge[](v1)(v2)
\Edge[](v1)(v5)
\Edge[](v1)(v6)
\Edge[](v1)(v9)
\Edge[](v1)(v10)
\Edge[](v1)(v11)
\Edge[](v1)(v12)
\Edge[](v1)(v15)
\Edge[](v2)(v3)
\Edge[](v2)(v7)
\Edge[](v2)(v8)
\Edge[](v2)(v10)
\Edge[](v2)(v13)
\Edge[](v2)(v14)
\Edge[](v3)(v16)
\Edge[](v3)(v4)
\Edge[](v3)(v6)
\Edge[](v3)(v7)
\Edge[](v3)(v9)
\Edge[](v3)(v10)
\Edge[](v3)(v12)
\Edge[](v3)(v15)
\Edge[](v4)(v5)
\Edge[](v4)(v8)
\Edge[](v4)(v9)
\Edge[](v4)(v11)
\Edge[](v4)(v14)
\Edge[](v4)(v15)
\Edge[](v5)(v6)
\Edge[](v5)(v7)
\Edge[](v5)(v8)
\Edge[](v5)(v13)
\Edge[](v6)(v8)
\Edge[](v6)(v11)
\Edge[](v6)(v14)
\Edge[](v7)(v9)
\Edge[](v7)(v10)
\Edge[](v7)(v11)
\Edge[](v7)(v14)
\Edge[](v7)(v15)
\Edge[](v8)(v16)
\Edge[](v8)(v9)
\Edge[](v8)(v10)
\Edge[](v8)(v12)
\Edge[](v8)(v15)
\Edge[](v9)(v13)
\Edge[](v9)(v15)
\Edge[](v10)(v13)
\Edge[](v10)(v14)
\Edge[](v11)(v16)
\Edge[](v11)(v12)
\Edge[](v11)(v13)
\Edge[](v12)(v16)
\Edge[](v12)(v13)
\Edge[](v12)(v14)
\Edge[](v13)(v16)
\Edge[](v13)(v14)
\Edge[](v13)(v15)
\Edge[](v14)(v16)
\end{tikzpicture}}
\scalebox{\c}{\qquad\begin{tikzpicture}
\GraphInit[vstyle=Classic]
\Vertex[L=\hbox{},x=\r*0.4647231720437686cm,y=\r*0.8854560256532099cm]{v0}
\Vertex[L=\hbox{},x=\r*0.0cm,y=\r*1.0cm]{v1}
\Vertex[L=\hbox{},x=\r*-0.4647231720437686cm,y=\r*0.8854560256532099cm]{v2}
\Vertex[L=\hbox{},x=\r*-0.8229838658936564cm,y=\r*0.5680647467311558cm]{v3}
\Vertex[L=\hbox{},x=\r*-0.992708874098054cm,y=\r*0.12053668025532306cm]{v4}
\Vertex[L=\hbox{},x=\r*-0.9350162426854148cm,y=\r*-0.3546048870425356cm]{v5}
\Vertex[L=\hbox{},x=\r*-0.6631226582407953cm,y=\r*-0.7485107481711011cm]{v6}
\Vertex[L=\hbox{},x=\r*-0.23931566428755782cm,y=\r*-0.970941817426052cm]{v7}
\Vertex[L=\hbox{},x=\r*0.23931566428755782cm,y=\r*-0.970941817426052cm]{v8}
\Vertex[L=\hbox{},x=\r*0.6631226582407953cm,y=\r*-0.7485107481711011cm]{v9}
\Vertex[L=\hbox{},x=\r*0.9350162426854148cm,y=\r*-0.3546048870425356cm]{v10}
\Vertex[L=\hbox{},x=\r*0.992708874098054cm,y=\r*0.12053668025532306cm]{v11}
\Vertex[L=\hbox{},x=\r*0.8229838658936564cm,y=\r*0.5680647467311558cm]{v12}
\Edge[](v0)(v1)
\Edge[](v0)(v4)
\Edge[](v0)(v5)
\Edge[](v0)(v6)
\Edge[](v1)(v2)
\Edge[](v1)(v5)
\Edge[](v1)(v6)
\Edge[](v1)(v7)
\Edge[](v1)(v8)
\Edge[](v1)(v12)
\Edge[](v2)(v3)
\Edge[](v2)(v7)
\Edge[](v2)(v12)
\Edge[](v3)(v4)
\Edge[](v3)(v6)
\Edge[](v3)(v7)
\Edge[](v3)(v8)
\Edge[](v3)(v12)
\Edge[](v4)(v5)
\Edge[](v4)(v8)
\Edge[](v5)(v6)
\Edge[](v6)(v9)
\Edge[](v6)(v10)
\Edge[](v6)(v11)
\Edge[](v7)(v12)
\Edge[](v8)(v9)
\Edge[](v8)(v10)
\Edge[](v8)(v11)
\Edge[](v9)(v10)
\Edge[](v9)(v11)
\Edge[](v10)(v11)
\end{tikzpicture}}
\scalebox{\c}{\qquad\begin{tikzpicture}
\GraphInit[vstyle=Classic]
\Vertex[L=\hbox{},x=\r*0.6427876096865394cm,y=\r*0.766044443118978cm]{v0}
\Vertex[L=\hbox{},x=\r*0.0cm,y=\r*1.0cm]{v1}
\Vertex[L=\hbox{},x=\r*-0.6427876096865394cm,y=\r*0.766044443118978cm]{v2}
\Vertex[L=\hbox{},x=\r*-0.984807753012208cm,y=\r*0.17364817766693033cm]{v3}
\Vertex[L=\hbox{},x=\r*-0.8660254037844386cm,y=\r*-0.5cm]{v4}
\Vertex[L=\hbox{},x=\r*-0.3420201433256688cm,y=\r*-0.9396926207859083cm]{v5}
\Vertex[L=\hbox{},x=\r*0.3420201433256688cm,y=\r*-0.9396926207859083cm]{v6}
\Vertex[L=\hbox{},x=\r*0.8660254037844386cm,y=\r*-0.5cm]{v7}
\Vertex[L=\hbox{},x=\r*0.984807753012208cm,y=\r*0.17364817766693033cm]{v8}
\Edge[](v0)(v1)
\Edge[](v0)(v4)
\Edge[](v0)(v5)
\Edge[](v0)(v6)
\Edge[](v1)(v2)
\Edge[](v1)(v5)
\Edge[](v1)(v6)
\Edge[](v1)(v7)
\Edge[](v1)(v8)
\Edge[](v2)(v3)
\Edge[](v2)(v7)
\Edge[](v2)(v8)
\Edge[](v3)(v4)
\Edge[](v3)(v7)
\Edge[](v3)(v8)
\Edge[](v4)(v5)
\Edge[](v4)(v6)
\Edge[](v5)(v6)
\Edge[](v7)(v8)
\end{tikzpicture}}
\scalebox{\c}{\qquad\begin{tikzpicture}
\GraphInit[vstyle=Classic]
\Vertex[L=\hbox{},x=\r*0.6427876096865394cm,y=\r*0.766044443118978cm]{v0}
\Vertex[L=\hbox{},x=\r*0.0cm,y=\r*1.0cm]{v1}
\Vertex[L=\hbox{},x=\r*-0.6427876096865394cm,y=\r*0.766044443118978cm]{v2}
\Vertex[L=\hbox{},x=\r*-0.984807753012208cm,y=\r*0.17364817766693033cm]{v3}
\Vertex[L=\hbox{},x=\r*-0.8660254037844386cm,y=\r*-0.5cm]{v4}
\Vertex[L=\hbox{},x=\r*-0.3420201433256688cm,y=\r*-0.9396926207859083cm]{v5}
\Vertex[L=\hbox{},x=\r*0.3420201433256688cm,y=\r*-0.9396926207859083cm]{v6}
\Vertex[L=\hbox{},x=\r*0.8660254037844386cm,y=\r*-0.5cm]{v7}
\Vertex[L=\hbox{},x=\r*0.984807753012208cm,y=\r*0.17364817766693033cm]{v8}
\Edge[](v0)(v1)
\Edge[](v0)(v4)
\Edge[](v0)(v5)
\Edge[](v0)(v6)
\Edge[](v0)(v7)
\Edge[](v0)(v8)
\Edge[](v1)(v2)
\Edge[](v1)(v5)
\Edge[](v1)(v6)
\Edge[](v1)(v7)
\Edge[](v1)(v8)
\Edge[](v2)(v3)
\Edge[](v2)(v5)
\Edge[](v2)(v7)
\Edge[](v3)(v4)
\Edge[](v3)(v6)
\Edge[](v3)(v8)
\Edge[](v4)(v5)
\Edge[](v4)(v7)
\Edge[](v5)(v7)
\Edge[](v5)(v8)
\Edge[](v6)(v7)
\end{tikzpicture}}
\scalebox{\c}{\qquad\begin{tikzpicture}
\GraphInit[vstyle=Classic]
\Vertex[L=\hbox{},x=\r*0.7818314824680298cm,y=\r*0.6234898018587335cm]{v0}
\Vertex[L=\hbox{},x=\r*0.0cm,y=\r*1.0cm]{v1}
\Vertex[L=\hbox{},x=\r*-0.7818314824680298cm,y=\r*0.6234898018587335cm]{v2}
\Vertex[L=\hbox{},x=\r*-0.9749279121818236cm,y=\r*-0.2225209339563144cm]{v3}
\Vertex[L=\hbox{},x=\r*-0.4338837391175582cm,y=\r*-0.9009688679024191cm]{v4}
\Vertex[L=\hbox{},x=\r*0.4338837391175582cm,y=\r*-0.9009688679024191cm]{v5}
\Vertex[L=\hbox{},x=\r*0.9749279121818236cm,y=\r*-0.2225209339563144cm]{v6}
\Edge[](v0)(v1)
\Edge[](v0)(v4)
\Edge[](v0)(v5)
\Edge[](v0)(v6)
\Edge[](v1)(v2)
\Edge[](v1)(v5)
\Edge[](v1)(v6)
\Edge[](v2)(v3)
\Edge[](v2)(v5)
\Edge[](v2)(v6)
\Edge[](v3)(v4)
\Edge[](v3)(v5)
\Edge[](v3)(v6)
\Edge[](v4)(v5)
\Edge[](v4)(v6)
\Edge[](v5)(v6)
\end{tikzpicture}}
    \caption{\ben{All 5-critical $(P_5,W_4)$-free graphs.}}
    \label{fig:all5critP5W4free}
\end{figure}
As mentioned in Section~\ref{6-vertex-critical}, each $k$-vertex-critical $(P_5,W_4)$-free graph is either isomorphic to $K_k$, isomorphic to $\overline{C_{2k-1}}$, contains $C_5$ as an induced subgraph or contains $\overline{C_7}$ as an induced subgraph. By additionally running\footnote{In fact, we made two independent implementations of \autoref{algo:genGraphs} and obtained the same results for both implementations, which gives us great confidence that there are no bugs leading to incorrect results. These implementations are made publicly available at~\cite{criticalpfree-site} and~\cite{jorik-github}.} \autoref{algo:genGraphs} for $k=5$, $\mathcal{H}=\{P_5,W_4\}$ and $I=C_5$ and $I=\overline{C_7}$ we obtain the characterization from \autoref{th:characterization}. We summarize the counts of $5$-vertex-critical $(P_5,W_4)$-free and $5$-critical $(P_5,W_4)$-free graphs (see \autoref{fig:all5critP5W4free}) in \autoref{tab:counts} and also make these graphs publicly available at the House of Graphs~\cite{CDG} at: \url{https://houseofgraphs.org/meta-directory/critical-h-free}. We also executed the algorithm for $k=6$, $\mathcal{H}=\{P_5,W_4\}$ and $I=C_5$, but the algorithm did not terminate after one week of running.

\begin{table}[htbp]
    \centering
    \footnotesize
    \begin{tabular}{|c| c | c |}
        \hline
        $n$ & \# $5$-vertex-critical $(P_5,W_4)$-free graphs & \# $5$-critical $(P_5,W_4)$-free graphs \\
        \hline
        5 & 1 & 1 \\
        6 & 0 & 0 \\
        7 & 1 & 1 \\
        8 & 1 & 1 \\
        9 & 44 & 7 \\
        10 & 4 & 1 \\
        11 & 0 & 0 \\
        12 & 1 & 1 \\
        13 & 8 & 6 \\
        14 & 0 & 0 \\
        15 & 2 & 1 \\
        16 & 0 & 0 \\
        17 & 2 & 2 \\
        \hline
        Total & 64 & 21\\
        \hline
    \end{tabular}
    \vskip12pt
    \caption{An overview of the number of pairwise non-isomorphic $5$-vertex-critical $(P_5,W_4)$-free graphs and $5$-critical $(P_5,W_4)$-free graphs.}
    \label{tab:counts}
\end{table}

\section{Conclusion}\label{conclusion}
In this paper, we proved that there are finitely many $k$-vertex-critical $(P_5, W_4)$-free graphs for all $k \ge 1$ and characterized all such graphs for $k=5$. Our results gave an affirmative
answer to the problem posed in~\cite{CGHS21} for $H = W_4$.
In the future, it is natural to investigate the finiteness of the set of $k$-vertex-critical $(P_5,H)$-free graphs for other graphs $H$ of order 5. \ben{From~\cite{CH23} and the recent work showing finiteness in the cases where $H=K_{1,3}+P_1$ and $H=\overline{K_3+2P_1}$~\cite{XJGH24}, the only remaining open cases are when $H$ is any of the following eight graphs of order $5$:}

    \begin{multicols}{2}
        \begin{itemize}
             \item \ben{co-gem} 
             \item \ben{chair} %(known $k=5$)
             \item \ben{$\overline{\text{diamond}+P_1}$} 
             \item \ben{$C_4+P_1$}
             \item \ben{bull} %(known $k=5$)
             \item  \ben{$\overline{P_3+2P_1}$}
             \item \ben{$K_5-e$} %(known $k\ge 8$)
             \item \ben{$K_5$} %(known $k=5$)
        \end{itemize}
    \end{multicols}

\red{Of these, we expect the results in this paper will be most directly applicable to $\overline{P_3+2P_1}$, $K_5-e$ and $K_5$ as these graphs are also dense, and so are more difficult to induce around a $C_5$. These are also the only remaining cases where the vertex-critical graphs are $\overline{C_{2\ell+1}}$-free for some $\ell$ (and in fact, we have either $\ell=4$ or $\ell=5$). We also note that cases when $H$ is co-gem~\cite{BC24} and chair, $\overline{\text{diamond}+P_1}$,  or bull~\cite{ABCK24} are known to be finite for all $k\ge 1$ when the graphs are further restricted to be $2P_2$-free.}

%--------------------------------------------------------------------------------------------------

\subsection*{Acknowledgements}

\noindent  Shenwei Huang is supported by the Natural Science Foundation of China (NSFC) under Grant 12171256 and 12161141006.
We also acknowledge the support of the joint FWO-NSFC scientific mobility project with grant number VS01224N. The research of Jan Goedgebeur was supported by Internal Funds of KU Leuven and an FWO grant with grant number G0AGX24N. Jorik Jooken is supported by a Postdoctoral Fellowship of the Research Foundation Flanders (FWO) with grant number 1222524N. The reseach of Ben Cameron was supported by the Natural Sciences and Engineering Research Council of Canada (NSERC), grants RGPIN-2022-03697 and DGECR-2022-00446.

\end{document}